\documentclass[11pt]{amsart}
\usepackage[a4paper,scale=.8,twoside=false]{geometry}

\usepackage[dvipsnames]{xcolor} 

\usepackage{accents} 
\usepackage{amssymb} 
\usepackage[british]{babel} 
\usepackage[bibstyle=ieee,citestyle=numeric-comp,dashed=true,sorting=nty,backref=true,maxbibnames=7,maxcitenames=3]{biblatex}
\DeclareSourcemap{
    \maps[datatype=bibtex]{
        \map{
            \step[fieldset=issn, null]
        }
    }
}
\DeclareNumChars*{+} 

\usepackage{bm} 
\usepackage{csquotes} 
\usepackage[shortlabels,inline]{enumitem} 
\usepackage{esint} 
\usepackage[T1]{fontenc} 
\usepackage{graphicx} 
\usepackage{lmodern} 
\usepackage{mathtools} 
\usepackage{microtype} 
\usepackage{multirow} 
\usepackage{subcaption} 

\usepackage[colorlinks=true,urlcolor=NavyBlue,citecolor=ForestGreen,linkcolor=BrickRed]{hyperref} 
\usepackage[nameinlink]{cleveref} 

\numberwithin{equation}{section}

\usepackage{pgfplots} 
\usepackage{tikz} 
\usetikzlibrary{external,calc}
\pgfplotsset{compat=1.17}

\tikzsetexternalprefix{tikz/}
\newcommand{\map}[3]{ #1 \colon #2 \to #3 }
\newcommand{\wildcard}{{\makebox[\widthof{$x$}][c]{$\cdot$}}}

\let\oldbm\bm
\renewcommand{\bm}[1]{{\oldbm{#1}}} 
\newcommand{\mb}[1]{{\mathbb{#1}}} 
\newcommand{\mc}[1]{{\mathcal{#1}}} 
\newcommand{\mf}[1]{{\mathfrak{#1}}} 
\newcommand{\ms}[1]{{\mathsf{#1}}} 

\newcommand{\bmf}[1]{{\oldbm{\mathfrak{#1}}}} 
\newcommand{\bms}[1]{{\oldbm{\mathsf{#1}}}} 
\newcommand{\obm}[1]{{\overline{\oldbm{#1}}}} 

\newcommand{\mo}[1]{{\mathring{#1}}}
\newcommand{\ol}[1]{{\overline{#1}}}
\newcommand{\ul}[1]{{\underline{#1}}}
\newcommand{\ac}[2]{\accentset{#1}{#2}}
\newcommand{\md}[1]{\ac{\diamond}{#1}}

\newcommand{\ceq}{\coloneqq}
\newcommand{\eqc}{\eqqcolon}

\newcommand{\R}{ \mb{R} }

\newcommand{\Z}{ \mb{Z} }
\newcommand{\N}{ \mb{N} }

\MHInternalSyntaxOn
\renewcommand\MT_delim_default_inner_wrappers:n [1]{
    \@namedef{MT_delim_\MH_cs_to_str:N #1 _star_wrapper:nnn}##1##2##3{
        \mathopen{}\mathclose\bgroup ##1 ##2 \aftergroup\egroup ##3
    }
    \@namedef{MT_delim_\MH_cs_to_str:N #1 _nostarscaled_wrapper:nnn}##1##2##3{
        \mathopen{##1}##2\mathclose{##3}
    }
    \@namedef{MT_delim_\MH_cs_to_str:N #1 _nostarnonscaled_wrapper:nnn}##1##2##3{
        \ifx.##1\else\mathopen##1\fi##2\ifx.##3\else\mathclose##3\fi
    }
}
\MHInternalSyntaxOff

\DeclarePairedDelimiter{\parn}{\lparen}{\rparen} 
\DeclarePairedDelimiter{\brac}{\lbrace}{\rbrace} 
\DeclarePairedDelimiter{\brak}{\lbrack}{\rbrack} 
\DeclarePairedDelimiter{\angl}{\langle}{\rangle} 

\DeclarePairedDelimiter{\abs}{\lvert}{\rvert} 
\DeclarePairedDelimiter{\norm}{\lVert}{\rVert} 

\DeclarePairedDelimiter{\floor}{\lfloor}{\rfloor}

\DeclarePairedDelimiter{\rest}{.}{\rvert}

\DeclareMathOperator{\id}{id}
\DeclareMathOperator{\lspan}{span}
\DeclareMathOperator{\ran}{ran}

\DeclareMathOperator*{\lesssimop}{\lesssim}
\DeclareMathOperator*{\equalop}{=}

\theoremstyle{plain}
\newtheorem{theorem}{Theorem}[section]
\newtheorem{corollary}[theorem]{Corollary}
\newtheorem{lemma}[theorem]{Lemma}
\newtheorem{proposition}[theorem]{Proposition}

\theoremstyle{definition}

\theoremstyle{remark}
\newtheorem{remark}[theorem]{Remark}

\crefname{enumi}{part}{parts} 

\newcommand{\loc}{\textnormal{loc}}

\title{Symmetric doubly periodic gravity-capillary waves with small vorticity}

\author[D. S. Seth]{Douglas S. Seth}
\address{Department of Mathematical Sciences, Norwegian University of Science and Technology, 7491 Trondheim, Norway}
\email{douglas.s.seth@ntnu.no}

\author[K. Varholm]{Kristoffer Varholm}
\address{Department of Mathematical Sciences, Norwegian University of Science and Technology, 7491 Trondheim, Norway}
\email{kristoffer.varholm@ntnu.no}

\author[E. Wahl\'en]{Erik Wahl\'en}
\address{Centre for Mathematical Sciences, Lund University, PO Box 118, 221 00 Lund, Sweden}
\email{erik.wahlen@math.lu.se}

\thanks{The authors are grateful to Josh Burby and Miles Wheeler for helpful discussions about~\cite{Lortz70Ueber}. This project has received funding from the European Research Council (ERC) under the European Union's Horizon 2020 research and innovation programme (grant agreement no. 678698) and the Swedish Research Council (grant no. 2020-00440).}

\begin{document}

\begin{abstract}
    We construct small amplitude gravity-capillary water waves with small nonzero vorticity, in three spatial dimensions, bifurcating from uniform flows. The waves are symmetric, and periodic in both horizontal coordinates. The proof is inspired by Lortz' construction of magnetohydrostatic equilibria in reflection-symmetric toroidal domains~\cite{Lortz70Ueber}. It relies on a global representation of the vorticity as the cross product of two gradients, and on prescribing a functional relationship between the Bernoulli function and the orbital periods of the water particles.

    The presence of the free surface introduces significant new challenges. In particular, the resulting free boundary problem is not elliptic, and the involved maps incur a loss of regularity under Fréchet differentiation. Nevertheless, we show that a version of the Crandall--Rabinowitz local bifurcation method still applies in this setting, by carefully tracking the loss of regularity.
\end{abstract}
\maketitle

\section{Introduction}
\subsection{The steady water wave problem}
The gravity-capillary water wave problem concerns the flow of an inviscid, incompressible fluid under the influence of gravity and surface tension. The fluid domain is bounded below by a rigid horizontal bottom $\brac{x_3=-d}$ and above by a free surface $\brac{x_3=\eta(t, \bm{x}')}$. We consider steady waves travelling with constant speed $c$ along the horizontal $x_1$-axis. That is, waves for which the dependence on $x_1$ and $t$ is through $x_1 - ct$ alone.

Using a frame of reference moving with the wave, we can assume that the velocity field $\bm{u}(\bm{x})$ and the surface $\eta(\bm{x}')$ are independent of time. The fluid domain is then given by
\begin{align*}
    \Omega^\eta & \ceq \brac{\bm{x} \in \R^3: -d < x_3 < \eta(\bm{x}')}, \\
    \shortintertext{with boundary $\partial \Omega^\eta = S^\eta \cup B$, where}
    S^\eta      & \ceq \brac{\bm{x} \in \R^3: x_3 = \eta(\bm{x}')}       \\
    \shortintertext{is the top surface and}
    B           & \ceq \brac{\bm{x} \in \R^3: x_3 = -d}
\end{align*}
the flat bottom. Here we use the notation $\bm{x}=(\bm{x}',x_3)=(x_1, x_2, x_3)$, while $\bm{u}=(u_1, u_2, u_3)$.

Assuming that the fluid has constant, unit density, the motion is governed by the incompressible Euler equations
\begin{subequations}
    \begin{gather}
        (\bm{u} \cdot \nabla)\bm{u} + \nabla p + g \bm{e}_3= 0,
        \label{eq:conservation_of_momentum}\\
        \nabla \cdot \bm{u} =0
    \end{gather}
\end{subequations}
in $\Omega^\eta$, where $p$ is the pressure and $g>0$ the constant gravitational acceleration. The corresponding kinematic boundary conditions read
\begin{subequations}
    \label{eq:kinematic}
    \begin{alignat}{-1}
        \bm{u} \cdot\bm{n} & = 0 & \qquad & \text{on $S^\eta$},
        \label{eq:kinematic_surface}                            \\
        u_3                & = 0 &        & \text{on $B$},
        \label{eq:kinematic_bed}
    \end{alignat}
\end{subequations}
where
\begin{equation}
    \label{eq:unit_normal}
    \bm{n}[\eta] \ceq \frac{(-\nabla \eta,1)}{\sqrt{1+\abs{\nabla \eta}^2}}
\end{equation}
is the upward-pointing unit normal on the surface. Their interpretation is simply that there is no flow through the boundary.

Finally, the dynamic boundary condition is the requirement that
\[
    p = \sigma \nabla \cdot \bm{n} \qquad \text{on $S^\eta$,}
\]
where $\sigma\ge 0$ is the coefficient of surface tension. One speaks of \emph{pure gravity waves} when $\sigma = 0$, and of \emph{gravity-capillary waves} when $\sigma > 0$. We will always assume the latter in this work, as the surface tension plays a crucial part.

We will concentrate on three-dimensional solutions, in the sense that the velocity field depends in a nontrivial way on all three spatial variables, while the surface pattern is truly two-dimensional. More specifically, we look for solutions that are \emph{doubly periodic}; that is, they are periodic in two distinct horizontal directions. Given wavelengths $\lambda_1, \lambda_2>0$, introduce the dual pair of lattices $\Lambda$ and $\Lambda^*$ through
\[
    \Lambda \ceq \lambda_1 \Z \times \lambda_2 \Z, \qquad \Lambda^* = \kappa_1 \Z \times \kappa_2 \Z,
\]
where $\lambda_i \kappa_i = 2\pi$. For convenience, we also let
\[
    \Omega_0^\eta \ceq \brac{\bm{x} \in \Omega^\eta : \bm{x}' \in (0,\lambda_1) \times (0,\lambda_2)}
\]
denote the fundamental domain. We will look for (horizontally) $\Lambda$-periodic solutions to the Euler equations. In fact, we will show a posteriori that the the solutions are actually periodic with respect to the finer diamond lattice
\begin{equation}
    \label{eq:lambda_diamond}
    \md{\Lambda}\ceq \brac[\bigg]{m\frac{\bm{\lambda}}{2}+n\frac{\bm{\lambda}^*}{2} : m,n\in \Z},
\end{equation}
where $\bm{\lambda} \ceq (\lambda_1,\lambda_2)$ and $\bm{\lambda}^* \ceq (\lambda_1,-\lambda_2)$.

We also impose the symmetry conditions
\begin{equation}
    \label{eq:u_eta_symmetry}
    \begin{aligned}
        \bm{u}(R_j) & = (-1)^j R_j \bm{u}, \\
        \eta(R_j)   & =\eta
    \end{aligned}
\end{equation}
for $j = 1,2$, where $R_j$ denotes reflection of the $j$th coordinate. More generally, we say that a vector field $\bm{f}$ is $(\pm)$-symmetric if it satisfies
\[
    \bm{f}(R_j) = \pm(-1)^j R_j\bm{f},
\]
and that a scalar field $\varphi$ is $(\pm)$-symmetric if
\[
    \varphi(R_j) = (\pm 1)^j \varphi.
\]
In particular,~\eqref{eq:u_eta_symmetry} entails that we look for $(+)$-symmetric $\bm{u}$ and $\eta$. Note that the curl operator flips the symmetry of a vector field, so that the \emph{vorticity}
\[
    \bm{\omega}\ceq \nabla \times \bm{u}
\]
is $(\mp)$-symmetric if $\bm{u}$ is $(\pm)$-symmetric. Moreover, if $\varphi$ is $(-)$-symmetric, then $\nabla \varphi$ is $(+)$-symmetric.

In contrast to most other studies of three-dimensional steady water waves (see e.g.~\cites{Craig00Travelling,Groves01Spatial,Groves07Three, Iooss09Small,Iooss11Asymmetrical, Reeder81Three, Sun93Three}), we do \emph{not} assume that the velocity field is irrotational. In fact, we explicitly seek to construct solutions with nonzero vorticity. This means, in particular, that $\bm{u}$ is not the gradient of a velocity potential. While the two-dimensional water wave problem with vorticity has received considerable attention in recent years (see e.g.~\cites{Constantin11Nonlinear,Constantin16Global, Strauss10Steady, Varvaruca16Recent, Haziot22Traveling} and references therein), the literature on the corresponding three-dimensional problem is very scarce. Recently, \textcite{Lokharu20Existence} proved the existence of small-amplitude doubly periodic solutions under the assumption that $\bm{u}$ is \emph{Beltrami}; or, more specifically, that $\bm{\omega}=\alpha \bm{u}$ for some constant $\alpha$. An alternative existence proof with a discussion of pure gravity waves can be found in~\cite{Groves23Analytical}, and an extension to internal waves in~\cite{Seth23Internal}. Apart from this, there is a non-existence result for steady three-dimensional waves in the case of constant $\bm{\omega}$~\cite{Wahlen14Non} (see also~\cite{Constantin11Two, Martin22Three, Chen23Rigidity} and references therein for related results), and some existence results for explicit Gerstner-type solutions in the setting of edge waves or geophysical fluid dynamics; see e.g.~\cites{Constantin01Edge, Constantin12Exact, Henry18Three} and references therein. In this work, we adopt a new approach.

\subsection{Lortz' ansatz}
Our existence result is inspired by an ansatz that \textcite{Lortz70Ueber} used to construct symmetric magnetohydrostatic equilibria in toroidal domains, with fixed boundaries. The fact that this idea can be used in the context of water waves follows from the well-known correspondence between the steady Euler equations and the magnetohydrostatic equations; where one identifies the velocity field with the magnetic field, and the vorticity with the current density. For simplicity, we explain Lortz' ansatz in the context of fluid mechanics.

By the identity
\[
    (\bm{u} \cdot \nabla)\bm{u}+ \bm{u} \times \bm{\omega} = \nabla\parn[\bigg]{\frac{1}{2}\abs{\bm{u}}^2}
\]
we have that~\eqref{eq:conservation_of_momentum} is equivalent to
\begin{equation}
    \label{eq:Bernoulli}
    \nabla H = \bm{u} \times \bm{\omega},
\end{equation}
where
\[
    H \ceq \frac{1}{2}\abs*{\bm{u}}^2 + p + gx_3
\]
is the so-called \emph{Bernoulli function}. Since
\[
    (\bm{u} \times \bm{\omega}) \cdot \bm{u} = (\bm{u} \times \bm{\omega}) \cdot \bm{\omega}= 0,
\]
this function is necessarily constant on both streamlines and vortex lines.

Recall also that a divergence free vector field can locally be written as the cross product of two gradients (see e.g.~\cite{Barbarosie11Representation} and references therein), which in particular applies to the vorticity $\bm{\omega}$. Part of Lortz' ansatz is to assume that this can be done globally, with one of the functions being the Bernoulli function.
In other words, we suppose that
\begin{equation}
    \label{eq:Clebsch}
    \bm{\omega} =\nabla H \times \nabla \tau
\end{equation}
for two potentials $H$ and $\tau$. At this stage, we momentarily forget about~\eqref{eq:Bernoulli}, and suppose that $H$ and $\tau$ are arbitrary. However, since
\begin{align*}
    \bm{u} \times \bm{\omega} & = \bm{u} \times (\nabla H \times \nabla \tau)                              \\
                              & = (\bm{u} \cdot \nabla \tau)\nabla H - (\bm{u} \cdot \nabla H) \nabla \tau
\end{align*}
under~\eqref{eq:Clebsch}, we recover~\eqref{eq:Bernoulli} if we further assume that
\begin{align}
    \bm{u} \cdot \nabla H    & = 0
    \label{eq:H_transport}          \\
    \shortintertext{and}
    \bm{u} \cdot \nabla \tau & = 1,
    \label{eq:tau_transport}
\end{align}
respectively, and in turn~\eqref{eq:conservation_of_momentum} if we let
\[
    p \ceq H - \frac{1}{2}\abs{\bm{u}}^2 - gx_3 + Q
\]
for some constant $Q$. It is important to note here that~\eqref{eq:Clebsch} is an implicit equation for $\bm{u}$, because $H$ and $\tau$ depend on $\bm{u}$ through~\eqref{eq:H_transport} and \eqref{eq:tau_transport}.

Let us now \emph{define} $\tau$ through the transport equation~\eqref{eq:tau_transport}, with the boundary condition
\[
    \rest*{\tau}_{x_1 = 0} = 0,
\]
which uniquely determines $\tau$ as long as $u_1$ has a definite sign.
The value $\tau(\bm{x})$ has an interpretation as the time it takes to travel to $\bm{x}$ from a point on the symmetry plane $\brac{x_1=0}$, along the streamline passing through $\bm{x}$. See \Cref{fig:time_functions}. In turn, we use $\tau$ to introduce
\[
    q \ceq \tau(\wildcard+\lambda_1\bm{e}_1) - \tau,
\]
which satisfies
\[
    \bm{u} \cdot \nabla q = 0
\]
by periodicity of $\bm{u}$, and is therefore constant on the streamlines.

\begin{figure}
    \centering
    \tikzsetnextfilename{time_functions}
    \begin{tikzpicture}
        \begin{axis}[clip=false,width=\textwidth,height=0.25\textheight,hide axis,xmin=0,xmax=2*pi,ymin=-1,ymax=1,zmin=-1,zmax=1]
            \filldraw[black, draw opacity = 1.0, fill opacity = 0.1,text opacity = 1.0] (axis cs:0,-1,-1) node[left] {$x_1 = 0$} -- (axis cs:0,1,-1) -- (axis cs:0,1,1) -- (axis cs:0,-1,1) -- cycle;
            \addplot3[black,variable=t,domain=0:{pi/2},samples y = 1] (t,0, {0.5*(cos(deg(t))-1)});
            \addplot3[black,dashed,variable=t,domain={pi/2}:{2*pi},samples y = 1] (t,0, {0.5*(cos(deg(t))-1)});
            \filldraw[black, draw opacity = 1.0, fill opacity = 0.1,text opacity = 1.0] (axis cs:2*pi,-1,-1) node[left] {$x_1 = \lambda_1$} -- (axis cs:2*pi,1,-1) -- (axis cs:2*pi,1,1) -- (axis cs:2*pi,-1,1) -- cycle;
            \draw[black] (axis cs:0,0,0) circle (2 pt) node[above right] {$t = 0$};
            \fill[black] (axis cs:pi/2,0,-0.5) circle (2 pt) node[below] {$\bm{{x}}$} node[above right]{$t = \tau(\bm{x})$};
            \draw[black] (axis cs:2*pi,0,0) circle (2 pt) node[above right] {$t = q(\bm{x})$};
        \end{axis}
    \end{tikzpicture}
    \caption{The figure shows a streamline of $\bm{u}$ between the planes $x_1=0$ and $x_1=\lambda_1$. A particle travelling along the streamline starting on the plane $x_1=0$ at time $t=0$ reaches the point $\bm{x}$ on the streamline at time $t=\tau(\bm{x})$. The function $q$ is constant along the streamline and its value is the time it takes the particle to reach the plane $x_1=\lambda_1$.}
    \label{fig:time_functions}
\end{figure}

Under the periodicity and symmetry assumptions on $\bm{u}$, all the streamlines are periodic, and we can interpret $q(\bm{x})$ as the orbital period of the particle at $\bm{x}$. That is, the time required to travel from the plane $\brac{x_1=0}$ to the plane $\brac{x_1=\lambda_1}$ along the streamline passing through $\bm{x}$. Moreover, we will see in \Cref{prop:transport_equation_existence_uniqueness} that $q$ is both $\Lambda$-periodic and $(+)$-symmetric. Making the further ansatz that $H \ceq h(q)$ for some differentiable function $h$, we do indeed have
\[
    \bm{u} \cdot \nabla H = h'(q)(\bm{u} \cdot \nabla q) = 0,
\]
whence~\eqref{eq:H_transport} is automatically fulfilled.

Based on the above computations, we therefore look for solutions of
\begin{subequations}
    \label{eq:unflattened_problem}
    \begin{align}
        \label{eq:unflattened_euler I}
        \nabla \times \bm{u} & = h'(q) \nabla q \times \nabla \tau, \\
        \label{eq:unflattened_euler II}
        \nabla \cdot \bm{u}  & = 0
    \end{align}
    in $\Omega^\eta$, together with the kinematic boundary conditions~\eqref{eq:kinematic_surface} and \eqref{eq:kinematic_bed}, the dynamic boundary condition
    \begin{equation}
        \label{eq:dynamic_boundary_condition}
        \frac{1}{2}\abs{\bm{u}}^2 + g\eta + \sigma \nabla \cdot \bm{n}-h(q) = Q
    \end{equation}
    on $S^\eta$, and the integral condition
    \begin{equation}
        \label{eq:integral_cond}
        \fint_{\Omega_0^\eta} u_1\,d\bm{x} = c,
    \end{equation}
\end{subequations}
where $\fint$ denotes the mean. The number $c$ can be thought of as a wave speed, and will be used as a bifurcation parameter. In particular, we consider it part of the solutions rather than a fixed constant.

Lortz used the above ansatz, combined with a fixed point argument, to solve the magnetohydrostatic problem in a \emph{fixed} domain. As we shall see, the fact that we have a free boundary, with the additional boundary condition~\eqref{eq:dynamic_boundary_condition}, introduces significant challenges to overcome.

\subsection{Trivial solutions and dispersion relation}
If we \emph{choose} the constant on the right-hand side of~\eqref{eq:dynamic_boundary_condition} to be
\[
    Q = Q(c) \ceq \frac{1}{2}c^2- h\parn[\bigg]{\frac{\lambda_1}{c}},
\]
for each $c \neq 0$, then
\begin{equation}
    \label{eq:trivial_solutions}
    (\ac{0}{\bm{u}}[c],\ac{0}{\eta}) = (c \bm{e}_1,0)
\end{equation}
become trivial solutions of~\eqref{eq:unflattened_problem}, with corresponding time functions given by
\begin{equation}
    \label{eq:trivial_time_functions}
    \ac{0}{\tau}[c] = \frac{x_1}{c}, \qquad \ac{0}{q}[c] = \frac{\lambda_1}{c}.
\end{equation}

If we linearise the problem around such a trivial solution, and look for solutions of the form $\eta(\bm{x}')=e^{i \bm{x}' \cdot \bm{k}}$ for $\bm{k} \in \Lambda^* \setminus \brac{\bm{0}}$, we will obtain the linear dispersion relation
\begin{equation}
    \label{eq:disp_rel}
    \ell_{\bm{k}}(c) \ceq g+\sigma\abs{\bm{k}}^2-c^2k_1^2 \mo{\mf{f}}_{\abs{\bm{k}}} = 0, \quad \text{where} \quad \mo{\mf{f}}_{k} \ceq \frac{\coth(kd)}{k},
\end{equation}
which is the same as the one for the irrotational water-wave problem (see \Cref{sec:Dynamic BC}). Note that solutions to this equation always come in quadruples when $k_2 \neq 0$, due to symmetry: If $\bm{k}$ is a solution, then so are $-\bm{k}$ and $\pm \bm{k}^*$.

The multiplicity of solutions to~\eqref{eq:disp_rel} has been studied by many authors, including \textcite{Reeder81Three}, who gave a detailed analysis. For completeness, and since we are using different parameters, we include a result similar to theirs. Its proof is deferred to \Cref{sec:Dynamic BC}.

\begin{proposition}
    \label{prop:F_kernel}
    There exists a cocountable set $\Sigma(\bm{\kappa},g,d) \subset \R^+$ such that $\pm\bm{\kappa}$ and $\pm\bm{\kappa}^*$ are the only solutions to~\eqref{eq:disp_rel} when
    \[
        c = c^*(\sigma) = \frac{1}{\kappa_1}\parn[\bigg]{\frac{g+\sigma\abs{\bm{\kappa}}^2}{\mo{\mf{f}}_{\abs{\bm{\kappa}}}}}^{1/2}
    \]
    and $\sigma \in \Sigma$. Here $\bm{\kappa} = (\kappa_1,\kappa_2)$.
\end{proposition}

\subsection{Main result}
In order to state our main result, we first introduce some notation: When dealing with a mapping $T$ from some vector space $X$ (of any dimension) into a function space $Y$, we will typically use square brackets to denote the argument $x\in X$, so that $T[x]$ denotes a function in $Y$. We will refer to the space of all bounded $k$-linear maps from $X^k$ to $Y$ as $\mc{L}^k(X,Y)$, with $\mc{L}^0(X,Y)$ interpreted as simply $Y$.

For an open, non-empty set $\Omega \subset \R^n$, and $k\in \N_0\ceq \N \cup \brac{0}$, we let $C^k(\ol{\Omega})$ denote the class of functions $\map{f}{\ol{\Omega}}{\R}$ that are $k$ times continuously differentiable in $\Omega$, and whose derivatives are bounded and extend continuously to $\ol \Omega$. This is a Banach space when equipped with the usual norm
\[
    \norm{f}_k \ceq \sup_{\substack{j \leq k \\ \bm{x} \in \Omega}}{\norm{D^jf(\bm{x})}_{\mc{L}^j(\R^n,\R^n)}}.
\]
In addition, when $r\in(0,1)$, we let $C^{k+r}(\ol{\Omega})$ be the class of $f\in C^k(\ol{\Omega})$ whose derivative of order $k$ is $r$-Hölder. This is again a Banach space under the norm
\[
    \norm{f}_{k+r}\ceq \norm{f}_k + \sup_{\bm{x} \neq \bm{y} \in \Omega}{\frac{\norm{D^k f(\bm{x})-D^k f(\bm{y})}}{\abs{\bm{x}-\bm{y}}^r}}.
\]
We also use the notation $C^s_\loc(\ol{\Omega})$ for functions that belong to $C^s(\ol{U})$ for every bounded open $U\subset \Omega$, along with $C^s(\ol{\Omega}; \R^m)$ (or its local version) for vector fields.

\begin{theorem}
    \label{thm:main informal}
    Assume that $\sigma\in \Sigma$ and $c=c^*(\sigma)$, so that $\pm\bm{\kappa}$ and $\pm \bm{\kappa}^*$ are the only solutions of the dispersion equation~\eqref{eq:disp_rel}, and fix a non-integer $s>4$. Then, for every sufficiently small $h\in C^s(\R)$, there exists a family of nontrivial, $\md{\Lambda}$-periodic, $(+)$-symmetric solutions
    \[
        (\bm{u}[t],\eta[t],c[t]) \quad \in \quad C^{s}\parn[\big]{\ol{\Omega^{\eta[t]}}; \R^3}\times C^{s+1}(\R^2)\times \R
    \]
    for $0 < \abs{t} \ll 1$, to the problem~\eqref{eq:kinematic},~\eqref{eq:unflattened_problem}, bifurcating from the trivial solution $(\ac{0}{\bm{u}}[c^*], 0, c^*)$ at $t=0$. Moreover,
    \[
        \eta[t]=\ac{1}{\eta}t+o(t), \qquad \text{where} \qquad \ac{1}{\eta}(\bm{x}')\ceq\cos(\kappa_1 x_1)\cos(\kappa_2x_2),
    \]
    in $C^s(\R^2)$, and the solutions are rotational if $h'(\ac{0}{q}[c^*]) \neq 0$.
\end{theorem}

\begin{remark}
    From now on, $s$ will always refer to a non-integer in this paper.
\end{remark}

A more precise statement of the existence part of \Cref{thm:main informal}, with regularity properties of the solution branch with respect to the parameter $t$, is provided in \Cref{thm:Main_result}. There are two complications regarding this dependence: The first is that the domain $\Omega^{\eta[t]}$, in which $\bm{u}[t]$ is defined, depends on the parameter $t$. This is of course standard in free boundary problems, and is solved by considering a `flattened' version of the problem (see \Cref{sec:flattening}). The second is that, contrary to most studies of steady water waves, regularity with respect to $t$ comes at the price of losing regularity in $\bm{x}$. This phenomenon can be seen already in the expansion of the surface in \Cref{thm:main informal}, which is valid in $C^s$ rather than the expected $C^{s+1}$. The loss is related to the `elliptic-hyperbolic' nature of our problem, whereas most other versions of the steady water wave problem involve elliptic free boundary problems. In particular, there is a loss of derivatives in the dependence of the function $\tau$ on the vector field $\bm{u}$ (see \Cref{sec:transport}).

A somewhat similar phenomenon occurs in the irrotational three-dimensional pure gravity-wave problem~\cites{Iooss09Small,Iooss11Asymmetrical}. However, there the problem comes from the dynamic boundary condition, rather than the equations in the interior. Fortunately, our situation is more favourable; in the sense that we can solve the problem by an adaptation of the Crandall--Rabinowitz local bifurcation theorem, instead of needing to resort to a Nash--Moser scheme. This is accomplished with the implicit function and Banach fixed-point theorems, combined with carefully keeping track of the differentiability properties of fixed points. We need to do this when the involved mappings are only differentiable with a loss of regularity; see \Cref{Appendix:composition,Appendix:fixed_points}. The results in these appendices, and their proofs, draw inspiration from similar earlier results on scales of Banach spaces~\cites{Vanderbauwhede87Center,Hilger92Smoothness}.

The fact that the resulting solutions actually have nonzero vorticity is not completely obvious from the construction. We prove this after the fact in \Cref{prop:nonzero_vorticity}, by explicitly expanding the solution curve in terms of the bifurcation parameter $t$.

\subsection{Previous results}
We conclude this introduction with some historical remarks about doubly periodic steady water waves. The first rigorous construction of such waves in the irrotational setting is due to \textcite{Reeder81Three}, who considered periodic lattices of the same type as in~\eqref{eq:lambda_diamond}. Later, \textcite{Sun93Three} proved a similar result under the influence of pressure forcing. Waves with an arbitrary fundamental domain were constructed by \textcite{Craig00Travelling}, using a variational Lyapunov--Schmidt reduction and topological methods. Similar results can also be obtained using the spatial-dynamics method, in which the equations of motion are seen as an infinite-dimensional dynamical system where one of the horizontal spatial directions plays the role of time. Symmetric diamond waves were constructed by \textcite{Groves01Spatial}, while \textcite{Groves03Bifurcation} considered arbitrary fundamental domains (see also \textcite{Nilsson19Three} and \textcite{Ahmad23Resonant} for related results in other physical settings). As previously mentioned, the case of pure gravity waves is much more challenging, due to the appearance of small divisors. Nevertheless, existence theories for both symmetric and asymmetric doubly periodic waves have been established by \textcites{Iooss09Small,Iooss11Asymmetrical}.

Turning to water waves with vorticity, the result that is most closely related to \Cref{thm:main informal} is the construction of doubly periodic gravity-capillary waves on Beltrami flows by \textcite{Lokharu20Existence} (see also the recent alternative proof and discussion of pure gravity waves by \textcite{Groves23Analytical} and the extension to internal waves by \textcite{Seth23Internal}). Note that the interior equations are elliptic in the Beltrami setting, and therefore the problem can be solved by more standard bifurcation methods, like in the irrotational gravity-capillary problem. However, since the class of Beltrami fields is not invariant under the reflections $R_1$ and $R_2$, it does not make sense to look for $(+)$-symmetric solutions there. This was handled in~\cite{Lokharu20Existence} by using a multi-parameter bifurcation approach. The symmetry property is, on the other hand, a key part of our approach here. It is not even clear how to generalise Lortz' ansatz to a form that admits asymmetric solutions. The same symmetries $R_1$ and $R_2$ are also used in~\cite{Reeder81Three} in the reduction to a one-dimensional bifurcation problem.

\section{Reformulation of the problem using a flattening transformation}
\label{sec:flattening}

If we introduce the flattening transform $\map{\Pi}{\Omega^0}{\Omega^\eta}$ through
\[
    \Pi(\obm{x}) = \obm{x} + \pi(\obm{x})\bm{e}_3, \qquad \text{where} \qquad \pi(\obm{x}) = \eta(\obm{x}')\parn[\Big]{1+\frac{\ol{x}_3}{d}},
\]
then
\begin{equation}
    \label{eq:J}
    J \ceq D\Pi= \id_{\R^3} + \bm{e}_3 \otimes \nabla \pi,
\end{equation}
with
\begin{equation}
    \label{eq:pi_derivative}
    \nabla\pi(\obm{x}) = \parn[\bigg]{\nabla\eta(\obm{x}')\parn[\Big]{1+\frac{\ol{x}_3}{d}},\frac{\eta(\obm{x}')}{d}}.
\end{equation}
Consequently,
\begin{equation}
    \label{eq:rho_definition}
    \begin{aligned}
        \rho & \ceq \det(J) = 1 + \bm{e}_3 \cdot \nabla \pi \\
             & = 1 + \frac{\eta}{d},
    \end{aligned}
\end{equation}
and
\[
    J^{-1} = \id_{\R^3} - \frac{1}{\rho}\bm{e}_3 \otimes \nabla \pi.
\]

Given any scalar field $\varphi$ on $\Omega^\eta$, we define the corresponding \emph{flattened} scalar field $\ol{\varphi}$ on $\Omega^0$ through
\[
    \ol{\varphi} = \varphi(\Pi)(=\varphi \circ \Pi),
\]
while for any vector field $\bm{f}$ on $\Omega^\eta$, we define the flattened vector field $\obm{f}$ by
\begin{equation}
    \label{eq:flattening_of_vector_field}
    J\obm{f}=\rho \bm{f}(\Pi)
\end{equation}
instead. The next lemma describes the effect that this flattening has on the differential operators appearing in the Euler equations. Its proof is standard, and is thus omitted.

\begin{lemma}
    The gradient of a flattened scalar field is given by\footnote{We adopt the convention that $\nabla \varphi = (D\varphi)^\top$ is a column vector.}
    \begin{align}
        \nabla \ol{\varphi}                                         & = \frac{1}{\rho}J^\top J \ol{\nabla \varphi},
        \label{eq:gradient_flattening}                                                                              \\
        \shortintertext{while}
        \nabla \cdot \obm{f}                                        & = \rho \ol{\nabla \cdot \bm{f}}\notag         \\
        \shortintertext{and}
        \nabla \times \parn[\Big]{\frac{1}{\rho}J^{\top} J \obm{f}} & = \ol{\nabla \times \bm{f}}
        \label{eq:curl_flattening}
    \end{align}
    for vector fields. In particular, the property of being divergence free is conserved under flattening.
\end{lemma}

Since
\begin{align*}
    J\ol{\nabla q \times \nabla \tau} & = \rho (\nabla q \times \nabla \tau)(\Pi)                                               \\
                                      & =\rho \parn[\big]{J^{-\top}\nabla \ol{q}} \times \parn[\big]{J^{-\top}\nabla \ol{\tau}} \\
                                      & = J \parn[\big]{\nabla \ol{q} \times \nabla\ol{\tau}}
\end{align*}
by~\eqref{eq:flattening_of_vector_field} and \eqref{eq:gradient_flattening}, we see that~\eqref{eq:unflattened_euler I} and \eqref{eq:unflattened_euler II} read simply
\begin{align*}
    \nabla \times \parn*{\frac{1}{\rho}J^\top J \obm{u}} & = h'(\ol{q}) \nabla \ol{q} \times \nabla \ol{\tau}, \\
    \nabla \cdot \obm{u}                                 & = 0,
\end{align*}
in the flattened domain. We therefore need to characterise $\ol{\tau}$ in a way that only involves flattened quantities. Using~\eqref{eq:flattening_of_vector_field} and \eqref{eq:gradient_flattening}, we find
\[
    1 = \bm{u} \cdot \nabla \tau =\parn*{\frac{1}{\rho}J\obm{u}}\cdot\parn[\big]{J^{-\top}\nabla\ol{\tau}} = \frac{1}{\rho}\obm{u}\cdot \nabla \ol{\tau},
\]
whence~\eqref{eq:tau_transport} becomes
\[
    \obm{u}\cdot \nabla \ol{\tau} = \rho
\]
in $\Omega^0$.

We also need to consider the boundary conditions. Observe that
\begin{align*}
    \rest*{J}_{S^0} & = \id_{\R^3} + \bm{e}_3 \otimes (\nabla \eta,\eta/d) \\
    \rest*{J}_B     & = \id_{\R^3} + \bm{e}_3 \otimes (0,\eta/d),
\end{align*}
from~\eqref{eq:J} and \eqref{eq:pi_derivative}, from which we obtain the simpler kinematic boundary condition
\[
    \ol{u}_3= 0
\]
on all of $\partial\Omega^0$. Meanwhile, the dynamic boundary condition~\eqref{eq:dynamic_boundary_condition} becomes
\[
    \frac{1}{2\rho^2}\parn[\big]{\abs{\obm{u}'}^2 + (\nabla \eta \cdot \obm{u}')^2} + g\eta + \sigma \nabla \cdot \bm{n}[\eta] -h(\ol{q}) = Q(c)
\]
on $S^0$. Here we recall that $\obm{u} = (\obm{u}',0)$ on $S^0$ under the kinematic condition, and that $\bm{n}$ is the unit normal defined in~\eqref{eq:unit_normal}.

Finally, the integral condition in~\eqref{eq:integral_cond} reads
\[
    \frac{1}{\abs{\Omega_0^\eta}} \int_{\Omega_0^0} \ol{u}_1\,d\obm{x} = \frac{1}{\abs{\Omega_0^\eta}}\int_{\Pi(\Omega_0^0)} \bm{u} \cdot \bm{e}_1\,d\bm{x} = \fint_{\Omega_0^\eta} u_1 \,d\bm{x} = c,
\]
after a change of variables, where one observes that
\[
    \abs{\Omega_0^\eta} = \int_{\Omega_0^0} \rho\,d\obm{x}.
\]

In the remainder, we will not hesitate to drop the bars on flattened quantities. To summarise, the problem in the flattened variables is
\begin{subequations}
    \label{eq:flatproblem}
    \begin{alignat}{-1}
        \label{eq:flattened_euler}
        \nabla \times \parn[\bigg]{\frac{1}{\rho}J^\top J \bm{u}} & = h'(q) \nabla q \times \nabla \tau & \qquad &                       \\
        \nabla \cdot \bm{u}                                       & = 0                                 &        & \text{in $\Omega^0$,}
        \label{eq:flattened_divergence}                                                                                                  \\
        \bm{u}\cdot \nabla \tau                                   & = \rho                              &        &
        \label{eq:flattened_time}
    \end{alignat}
    with
    \begin{equation}
        \label{eq:flattened_time_boundary}
        \rest*{\tau}_{x_1 = 0} = 0,\qquad \text{and} \qquad q \ceq \tau(\wildcard+\lambda_1\bm{e}_1) - \tau,
    \end{equation}
    under the integral condition
    \begin{equation}
        \label{eq:integral_cond_flattened}
        \frac{1}{\abs{\Omega_0^\eta}} \int_{\Omega_0^0} u_1\,d\bm{x} = c,
    \end{equation}
    and the boundary conditions
    \begin{equation}
        \label{eq:flattened_kinematic}
        u_3= 0 \qquad \text{on $\partial \Omega^0$}
    \end{equation}
    and
    \begin{equation}
        \label{eq:dynamic_bdry_flattened}
        \frac{1}{2\rho^2}\parn[\big]{\abs{\bm{u}'}^2 + (\nabla \eta \cdot \bm{u}')^2} +g\eta + \sigma \nabla \cdot \bm{n}[\eta]-h(q)= Q(c) \qquad \text{on $S^0$}.
    \end{equation}
\end{subequations}

\section{Reduction to the boundary}

The goal of this section is to reduce problem~\eqref{eq:flatproblem} to a single equation on the flattened surface $S^0$. We will do this by solving equations~\eqref{eq:flattened_euler}--\eqref{eq:flattened_kinematic} for $\bm{u}$, and then substituting their solution into the dynamic boundary condition~\eqref{eq:dynamic_bdry_flattened}. We begin by analysing the general div-curl problem
\begin{alignat*}{-1}
    \nabla \times \parn[\Big]{\frac{1}{\rho}J^\top J \bm{f}} & = \bm{g} & \qquad & \multirow{2}{*}{in $\Omega^0$,} \\
    \nabla \cdot \bm{f}                                      & = 0      &        &                                 \\
    f_3                                                      & =0       &        & \text{on $\partial \Omega^0$}
\end{alignat*}
in \Cref{sec:div-curl}. After this, we analyse the linear transport equation
\[
    \bm{u}\cdot \nabla g=f,
\]
in \Cref{sec:transport}, which will allow us to properly define the time functions $\tau$ and $q$. Finally, we solve the system~\eqref{eq:flattened_euler}--\eqref{eq:flattened_kinematic} using a fixed point argument in \Cref{sec:contraction}.

Before embarking on this program, we define the function spaces which will be used in the analysis.
For any $s > 0$, we define the spaces
\begin{align*}
    \mb{X}^{s}_{\pm}      & \ceq \brac[\big]{\bm{f} \in C^{s} \parn[\big]{\ol{\Omega^0};\R^3}: \text{$\bm{f}$ is $\Lambda$-periodic, $(\pm)$-symmetric, $\nabla \cdot \bm{f} = 0$}}, \\
    \mo{\mb{X}}^{s}_{\pm} & \ceq \brac[\big]{\bm{f} \in \mb{X}^{s}_{\pm}: \text{$f_3=0$ on $\partial \Omega_0^0$}},                                                                  \\
    \mb{Y}^{s}_{\pm}      & \ceq \brac[\big]{\varphi \in C^{s} (\R^2): \text{$\varphi$ is $\Lambda$-periodic, $(\pm)$-symmetric}},
\end{align*}
and also the open sets
\begin{align*}
    \mc{U}_\delta^s   & \ceq \brac{\bm{f} \in \mo{\mb{X}}_+^s : f_1 > \delta},                              \\
    \mc{V}_\epsilon^s & \ceq \brac{\varphi \in \mb{Y}_+^{s} : \varphi + d > 0, \norm{\varphi}_s < \epsilon}
\end{align*}
where the velocity field and surface profile will reside. Here $\epsilon,\delta \ll 1$ are positive constants, with $\delta$ being fixed throughout. We may have to shrink $\epsilon$ at points, and we remind the reader that $s$ is always non-integral.

\subsection{Div-curl problem}
\label{sec:div-curl}
For all $s > 1$ and $\eta \in \mc{V}_\epsilon^{s+1}$, we define the bounded linear operator
\[
    \map{C[\eta]}{\mo{\mb{X}}_+^s}{\mb{X}^{s-1}_- \times \R}
\]
through
\begin{equation}
    \label{eq:C_definition}
    C[\eta]\bm{f}\ceq \parn[\bigg]{\nabla \times \parn[\bigg]{\frac{1}{\rho}J^\top J \bm{f}},\frac{1}{\abs{\Omega_0^\eta}}\int_{\Omega_0^0} f_1\,d\bm{x}},
\end{equation}
for which we have the following result.
\begin{proposition}
    \label{prop:div_curl_problem}
    The map
    \[
        \map{C}{\mc{V}_\epsilon^{s+1}}{\mc{L}\parn[\big]{\mo{\mb{X}}_+^s;\mb{X}^{s-1}_- \times \R}}
    \]
    is analytic, and $C[\eta]$ is an isomorphism for every $\eta \in \mc{V}_\epsilon^{s+1}$ when $\epsilon$ is sufficiently small.
\end{proposition}
\begin{proof}
    By continuity with respect to $\eta$, it is sufficient to show that $C[0]$ is an isomorphism. To that end, suppose first that $C[0]\bm{f}=0$ for some $\bm{f} \in \mo{\mb{X}}_+^s$. Then $\bm{f}$ satisfies
    \[
        \begin{aligned}
            \nabla \times \bm{f} & = 0 \\
            \nabla \cdot \bm{f}  & = 0
        \end{aligned}
        \quad
        \text{in $\Omega^0$,}
        \qquad
        f_3 = 0
        \quad
        \text{on $\partial\Omega^0$,}
        \quad
        \text{and}
        \quad
        \fint_{\Omega_0^0} f_1 \,d\bm{x} = 0,
    \]
    where taking the curl implies $\Delta \bm{f} = 0$. In particular, we must have $f_3 \equiv 0$ by uniqueness of the solution of the homogeneous Dirichlet problem. In turn,
    \[
        \nabla \times \bm{f} = \parn[\big]{-\partial_3 f_2, \partial_3 f_1, \partial_1 f_2 - \partial_2 f_1} = 0
    \]
    shows that $f_1, f_2$ are independent of $x_3$. They are therefore bounded harmonic functions on $\R^2$, and thus constant by Liouville's theorem. This constant vanishes by symmetry and the integral condition, whence $\bm{f} = 0$, and $C[0]$ is injective.

    For surjectivity, suppose that $\bm{g} \in \mb{X}^{s-1}_-$ and $c \in \R$. Standard theory for second order linear elliptic equations yields the existence of $\Lambda$-periodic and $(-)$-symmetric $\bm{h} \in C^{s+1} (\ol{\Omega^0};\R^3)$ such that
    \[
        \Delta \bm{h} = \bm{g} \quad \text{in $\Omega^0$} \qquad \text{and} \qquad h_1 = h_2 = \partial_3 h_3 = 0 \quad \text{on $\partial \Omega^0$,}
    \]
    and it is straightforward to check that
    \[
        \bm{f} \ceq c \bm{e}_1 - \nabla \times \bm{h} \in \mo{\mb{X}}_+^s
    \]
    satisfies $C[0]\bm{f} = (\bm{g},c)$.
\end{proof}

Due to \Cref{prop:div_curl_problem}, it makes sense to define $S[\eta] \in \mc{L}\parn{\mb{X}^{s-1}_-;\mo{\mb{X}}_+^s}$ and $\bmf{u}[\eta] \in \mo{\mb{X}}_+^s$ by
\begin{align*}
    S[\eta]\bm{g} & \ceq C[\eta]^{-1}(\bm{g},0), \\
    \bmf{u}[\eta] & \ceq C[\eta]^{-1}(0,1)
\end{align*}
when $\eta \in \mc{V}_\epsilon^{s+1}$, so that
\[
    C[\eta]^{-1}(\bm{g},c) = c \bmf{u}[\eta] + S[\eta]\bm{g},
\]
where we in particular note that $\bmf{u}[0] = \bm{e}_1$. The next result follows directly from~\eqref{eq:C_definition}.

\begin{lemma}
    \label{lemma:derivative_C}
    The derivative of $C$ at $\eta = 0$ is given by
    \[
        \angl{DC[0]\dot{\eta},\bm{f}} = \parn[\bigg]{\nabla \times (B[\dot{\eta}]\bm{f}),-\frac{1}{d}\fint_{\Omega_0^0} \dot{\eta}\,d\bm{x}\fint_{\Omega_0^0} f_1\,d\bm{x}},
    \]
    with
    \[
        B[\dot{\eta}] \ceq 2 \bm{e}_3 \odot \nabla \pi[\dot{\eta}] - \frac{1}{d}\dot{\eta}\id_{\R^3},
    \]
    where $\bm{x} \odot \bm{y} = (\bm{x} \otimes \bm{y} + \bm{y} \otimes \bm{x})/2$.
\end{lemma}

If $\dot{\eta}\in \mb{Y}_+^{s+1}$, then we can write
\[
    \dot{\eta}(\bm{x}') = \sum_{\bm{k} \in \Lambda^*} \dot{\eta}_{\bm{k}} e^{i \bm{k} \cdot \bm{x}'},
\]
where the coefficients are real and satisfy the symmetry relations
\[
    \dot{\eta}_{R_j\bm{k}} = \dot{\eta}_{\bm{k}}
\]
for $j=1,2$. We will use this to describe the derivative $D\bmf{u}[0]$.
\begin{proposition}
    \label{prop:bmf_derivative}
    One has
    \[
        D\bmf{u}[0] \dot{\eta} = \parn[\Big]{\frac{1}{d}\dot{\eta},0, - \partial_1 \pi[\dot{\eta}]} + \nabla\partial_1\chi[\dot{\eta}],
    \]
    for all $\dot{\eta} \in \mb{Y}_+^{s+1}$, where
    \[
        \chi \ceq\sum_{\bm{k} \in \Lambda^* \setminus \brac{\bm{0}}} \mf{f}_{\abs{\bm{k}}}\dot{\eta}_{\bm{k}} e^{i \bm{k} \cdot \bm{x}'}, \qquad \mf{f}_k(x_3) \ceq \frac{\cosh(k(x_3+d))}{k\sinh(k d)}
    \]
    makes $\bm{w}\ceq \nabla \partial_1\chi$ the unique solution to
    \begin{equation}
        \label{eq:Linearized_problem}
        \nabla \times \bm{w}=0 \quad \text{in $\Omega_0^0$}, \qquad w_3 =\partial_1 \pi[\dot{\eta}] \quad \text{on $\partial \Omega_0^0$}, \quad \text{and} \quad \int_{\Omega_0^0} w_1\,d\bm{x} = 0
    \end{equation}
    in $\mb{X}_+^s$.
\end{proposition}
\begin{proof}
    The problem in~\eqref{eq:Linearized_problem} is a special case of the one in~\cite[Section 3]{Lokharu20Existence}. By differentiating the identity $C[\eta]\bmf{u}[\eta] = (0,1)$, we find
    \[
        \angl{DC[0]\dot{\eta},\bm{e}_1} + C[0]D\bmf{u}[0]\dot{\eta} = 0,
    \]
    or
    \[
        D\bmf{u}[0]\dot{\eta} = C[0]^{-1}\parn[\bigg]{\nabla \times \parn[\bigg]{\frac{1}{d}\dot{\eta},0, - \partial_1 \pi[\dot{\eta}]},\frac{1}{d}\fint_{\Omega_0^0} \dot{\eta}\,d\bm{x}}
    \]
    after applying \Cref{lemma:derivative_C}. Now, clearly
    \[
        C[0]\parn[\bigg]{\overbrace{\parn[\bigg]{\frac{1}{d}\dot{\eta},0, - \partial_1 \pi[\dot{\eta}]} + \nabla\partial_1\chi}^{\in \mo{\mb{X}}_+^s}}=\parn[\bigg]{\nabla \times \parn[\bigg]{\frac{1}{d}\dot{\eta},0, - \partial_1 \pi[\dot{\eta}]},\frac{1}{d}\fint_{\Omega_0^0} \dot{\eta}\,d\bm{x}}. \qedhere
    \]
\end{proof}
As a consequence of the analyticity of $C$, we also get analyticity of $C^{-1}$. This, in turn, implies the following result, which we shall have need for later.
\begin{lemma}
    \label{lemma:S^su_and_derivatives_lipschitz}
    We have
    \begin{align*}
        \norm{D^k S[\eta_1]-D^k S[\eta_2]}_{\mc{L}^k\parn[\big]{\mb{Y}^{s+1}_+, \mc{L}\parn[\big]{\mb{X}^{s-1}_-,\mo{\mb{X}}_+^s}}} & \lesssim_{k,s} \norm{ \eta_1-\eta_2}_{s+1}, \\
        \norm{D^k \bmf{u}[\eta_1]-D^k \bmf{u}[\eta_2]}_{\mc{L}^k\parn[\big]{\mb{Y}^{s+1}_+,\mo{\mb{X}}_+^s}}                        & \lesssim_{k,s} \norm{ \eta_1-\eta_2}_{s+1}
    \end{align*}
    for all $\eta_1,\eta_2 \in \mc{V}_\epsilon^{s+1}$ and $k\in \N_0$.
\end{lemma}

\subsection{The steady linear transport equation}
\label{sec:transport}

We will now give a rigorous definition of the functions $\tau$ and $q$.
We begin with a general existence and uniqueness result for a linear transport equation.
\label{sec:transport_equation}

\begin{proposition}
    \label{prop:transport_equation_existence_uniqueness}
    Suppose that $s > 1$, and that $\bm{u}\in \mc{U}_\delta^s$. Then
    \begin{equation}
        \label{eq:transport_problem}
        \begin{cases}
            \bm{u} \cdot \nabla g= f, \\
            \rest{g}_{x_1=0}=0,
        \end{cases}
    \end{equation}
    has a unique solution $ g\in C_\loc^s(\ol{\Omega^0})$ for each $ f \in C_\loc^s(\ol{\Omega^0})$. Furthermore, for any $r_0 > 0$ we have
    \begin{equation}
        \label{eq:transport_bound}
        \norm{ g}_{s|r_0} \lesssim_{\norm{\bm{u}}_s,\delta,r_0} \norm{ f}_{s|r}
    \end{equation}
    for all $r \gg r_0$, where $\norm{\wildcard}_{s|r}$ is shorthand for the norm on $C^s(\ol{B}_r(0,\ol{\Omega^0}))$.

    Moreover, if
    \[
        f \in \mb{Y}_+^s(\ol{\Omega^0})\ceq \brac[\big]{\varphi \in C^{s} (\ol{\Omega^0}): \text{$\varphi$ is $\Lambda$-periodic, $(+)$-symmetric}},
    \]
    and $ G \ceq g(\wildcard + \lambda_1 \bm{e}_1) - g$, then
    \begin{enumerate}[(i)]
        \item $ g(\wildcard + \lambda_2 \bm{e}_2) = g$,
              \label{item:g_periodic_in_x2}
        \item $ g$ is $(-)$-symmetric,
              \label{item:g_symmetric}
        \item $ G \in \mb{Y}_+^s(\ol{\Omega^0})$, and
              \label{item:G_periodic_and_symmetric}
        \item $\nabla G \times \nabla g \in \mb{X}_-^{s-1}$,
              \label{item:rhs_periodic_and_symmetric}
    \end{enumerate}
    with $\norm{ G}_s \lesssim_{\norm{\bm{u}}_s,\delta} \norm{ f}_s$ and $\norm{\nabla G \times \nabla g}_{s-1} \lesssim_{\norm{\bm{u}}_s,\delta} \norm{ f}_s^2$.
\end{proposition}

\begin{proof}
    It is easily seen that if~\eqref{eq:transport_problem} has a solution, and if $\bm{\Phi}_{t,\mo{t}}(\bm{y})$ is the flow of $( u_2, u_3)/ u_1$ in the sense that
    \begin{equation}
        \label{eq:flow_definition}
        \begin{cases}
            \dot{\bm{\Phi}}_{t,\mo{t}}=( u_2/ u_1, u_3/ u_1)\parn[\big]{t,\bm{\Phi}_{t,\mo{t}}}, \\
            \bm{\Phi}_{\mo{t},\mo{t}}(\bm{y})=\bm{y},
        \end{cases}
    \end{equation}
    then this solution must be given by
    \begin{equation}
        \label{eq:transport_solution}
        g(\bm{x}) = \int_0^{x_1} \parn[\bigg]{ \frac{f}{u_1}}\parn[\big]{t,\bm{\Phi}_{t,x_1}(x_2, x_3)}\,dt,
    \end{equation}
    which yields both existence and uniqueness. As for regularity and the bound in~\eqref{eq:transport_bound}, note first that the vector field $(u_2/ u_1, u_3/ u_1)$ can be extended to a vector field in $C^s_\loc(\R^3; \R^2)$ by a variant of Whitney's extension theorem~\cite[Theorem 4, Chapter VI]{Stein70Singular}. Since the result we wish to prove is local in nature, we can even assume that this vector field is uniformly $C^s$. Now, an application of~\cite[Theorem A.6]{Eldering13Normally} yields that $\Phi_{t,\mo{t}}(\bm{y})$ is $C_\loc^s$ with respect to $(\mo{t},\bm{y})$ (locally) uniformly in $t$. Then, using~\eqref{eq:flow_definition}, we find that $\Phi$ is $C_\loc^s$ with respect to $(t,\mo{t}, \bm{y})$. Finally, an application of~\cite[Theorem 4.3]{Llave99Regularity} to~\eqref{eq:transport_solution} yields the regularity of $g$ and the bound in~\eqref{eq:transport_bound}.

    We immediately have \cref{item:g_periodic_in_x2} by uniqueness and the corresponding periodicity of $\bm{u}$ and $ f$. Similarly,
    \[
        \bm{u} \cdot \nabla ((-1)^j g(R_j)) = ((-1)^jR_j\bm{u}) \cdot \nabla g(R_j) = \bm{u}(R_j) \cdot \nabla g(R_j) = f(R_j) = f
    \]
    shows \cref{item:g_symmetric}. Apart from $ G$ inheriting~\ref{item:g_periodic_in_x2} and $ G(R_2)= G$, \cref{item:G_periodic_and_symmetric} is slightly less trivial, and it is simpler to appeal to the flow defined in~\eqref{eq:flow_definition}.

    By the properties $\bm{u}(\wildcard + \lambda_1 \bm{e}_1)= \bm{u}$ and $\bm{u}(R_1) = -R_1 \bm{u}$ one has
    \begin{align}
        \bm{\Phi}_{t + \lambda_1,\mo{t}+\lambda_1} & = \bm{\Phi}_{t,\mo{t}},
        \label{eq:flow_shift}                                                  \\
        \bm{\Phi}_{-t,0}                           & = \bm{\Phi}_{t,0}, \notag
    \end{align}
    through uniqueness of the flow, whence
    \[
        \bm{\Phi}_{\lambda_1,0} =\bm{\Phi}_{\lambda_1,\lambda_1/2} \circ \bm{\Phi}_{\lambda_1/2,0} = \bm{\Phi}_{0,-\lambda_1/2} \circ \bm{\Phi}_{-\lambda_1/2,0} = \id_{\R \times (-d,0)},
    \]
    and more generally $\bm{\Phi}_{t+\lambda_1,0} =\bm{\Phi}_{t,0}$, again appealing to uniqueness. In turn,
    \begin{equation}
        \label{eq:flow_periodicity}
        \bm{\Phi}_{t+\lambda_1,\mo{t}} = \bm{\Phi}_{t,\mo{t}+\lambda_1} = \bm{\Phi}_{t,\mo{t}}
    \end{equation}
    for all $t,\mo{t} \in \R$, generalizing~\eqref{eq:flow_shift}.

    Employing~\eqref{eq:transport_solution} and \eqref{eq:flow_periodicity}, we find that $ G$ is given by
    \[
        G(\bm{x}) =\int_0^{\lambda_1} \parn*{\frac{f}{u_1}}\parn[\big]{t-x_1,\bm{\Phi}_{t-x_1,x_1}(x_2, x_3)}\,dt,
    \]
    from which it follows, by periodicity of the integrand, that $ G(\wildcard+\lambda_1 \bm{e}_1) = G$. Using this, we also see that
    \begin{align*}
        G(R_1) & = g(R_1\wildcard + \lambda_1 \bm{e}_1)- g(R_1) = g(R_1(\wildcard -\lambda_1 \bm{e}_1))- g(R_1) \\
               & = G(\wildcard - \lambda_1 \bm{e}_1) = G,
    \end{align*}
    and therefore that $ G \in \mb{Y}_+^s(\ol{\Omega^0})$.

    For \cref{item:rhs_periodic_and_symmetric}, define $ \bm{v} \ceq \nabla G \times \nabla g$. The periodicity with respect to $x_2$ is immediate, and we see that
    \[
        \bm{v}(\wildcard + \lambda_1 \bm{e}_1) - \bm{v} = \nabla G \times \nabla G = 0
    \]
    by definition and periodicity of $ G$, and
    \[
        \bm{v}(R_j) = \nabla G(R_j) \times \nabla g(R_j) = R_j \nabla G \times (-1)^jR_j \nabla g = -(-1)^j \bm{v}
    \]
    by the symmetries of $g$ and $G$. Finally
    \[
        \nabla \cdot \bm{v} = (\nabla \times \nabla G)\cdot \nabla g - (\nabla \times \nabla g)\cdot \nabla G = 0. \qedhere
    \]
\end{proof}

\Cref{prop:transport_equation_existence_uniqueness} tells us that it makes sense to define a solution map $\map{g}{\mc{U}_\delta^s \times C^s_\loc(\ol{\Omega^0})}{C^s_\loc(\ol{\Omega^0})}$ through~\eqref{eq:transport_problem}, for any $s > 1$. Note that this map is linear with respect to $f$, and bounded in the sense of~\eqref{eq:transport_bound}. We will require regularity properties of this map, where the key observation
is that $ g$ satisfies
\begin{equation}
    \label{eq:g_f_difference}
    g[\bm{u}, f_1] - g[\bm{u}, f_2] = g[\bm{u}, f_1- f_2]
\end{equation}
when $s > 1$, $\bm{u} \in \mc{U}_\delta^s$, $f_1,f_2 \in C_\loc^s(\ol{\Omega^0})$, and
\begin{equation}
    \label{eq:g_u_difference}
    g[\bm{u}_1, f] - g[\bm{u}_2, f] = g\brak[\big]{\bm{u}_1,-(\bm{u}_1-\bm{u}_2)\cdot \nabla g[\bm{u}_2, f]}
\end{equation}
when $s > 2$, $u_1 \in \mc{U}_\delta^{s-1}$, $u_2 \in \mc{U}_\delta^s$ and $f \in C_\loc^s(\ol{\Omega^0})$. As an immediate consequence of~\eqref{eq:g_f_difference} and \eqref{eq:g_u_difference}, we obtain the following, where all bounds must be understood in the same sense as that in~\eqref{eq:transport_bound}.

\begin{proposition}[Lipschitz continuity]
    \label{prop:g_lipschitz_continuity}
    When $s > 1$, one has
    \[
        \norm{g[\bm{u}, f_1] - g[\bm{u}, f_2]}_s \lesssim_{\norm{\bm{u}}_s,\delta} \norm{ f_1 - f_2}_s
    \]
    for all $\bm{u} \in \mc{U}_\delta^s$, $f_1,f_2 \in C_\loc^s(\ol{\Omega^0})$, while
    \[
        \norm{ g[\bm{u}_1, f]- g[\bm{u}_2, f]}_{s-1} \lesssim_{\norm{\bm{u}_1}_{s-1},\norm{\bm{u}_2}_s,\delta} \norm{\bm{u}_1-\bm{u}_2}_{s-1} \norm{ f}_s
    \]
    for all $u_1 \in \mc{U}_\delta^{s-1}$, $u_2 \in \mc{U}_\delta^s$ and $f \in C_\loc^s(\ol{\Omega^0})$ when $s > 2$.
\end{proposition}

The solution map has the partial derivative
\begin{equation}
    \label{eq:g_derivative_f}
    \partial_fg[\bm{u},f]h = g[\bm{u},h]
\end{equation}
with respect to $f$; for every $s > 1$, $u \in \mc{U}_\delta^{s}$, and $f,h \in C_\loc^s(\ol{\Omega^0})$. Similarly, one has the \emph{formal} partial derivative
\begin{equation}
    \label{eq:g_derivative_u}
    \partial_{\bm{u}}g[\bm{u},f]\bm{v} \ceq -g\brak[\big]{\bm{u},\bm{v} \cdot \nabla g[\bm{u},f]}
\end{equation}
with respect to $\bm{u}$, for every $s > 2$, $\bm{u} \in \mc{U}_\delta^s$, $f \in C_\loc^s(\ol{\Omega^0})$, and $\bm{v} \in \mo{\mb{X}}_+^{s-1}$.

\begin{proposition}[First derivatives]
    \label{prop:g_first_derivative_u}
    We have
    \begin{align*}
        \norm{\partial_fg[\bm{u},f]}_{\mc{L}(C_\loc^s(\ol{\Omega^0}),C_\loc^s(\ol{\Omega^0}))}        & \lesssim_{\norm{\bm{u}}_s,\delta} 1,         \\
        \norm{\partial_{\bm{u}}g[\bm{u},f]}_{\mc{L}(\mo{\mb{X}}_+^{s-1},C_\loc^{s-1}(\ol{\Omega^0}))} & \lesssim_{\norm{\bm{u}}_s,\delta} \norm{f}_s
    \end{align*}
    for the operators defined in~\eqref{eq:g_derivative_f} and \eqref{eq:g_derivative_u}, respectively. The latter operator is a derivative in the sense that, if $s>3$, then
    \begin{equation}
        \label{eq:g_differentiability_u}
        \norm[\big]{\Delta_{(\bm{v},0)}g[\bm{u},f]-\partial_{\bm{u}}g[\bm{u},f]\bm{v}}_{s-2} \lesssim_{\norm{\bm{u}}_s, \norm{\bm{v}}_s,\delta} \norm{f}_s\norm{\bm{v}}_{s-1}^{2}
    \end{equation}
    for all $\bm{u},\bm{u}+\bm{v} \in \mc{U}_\delta^s$, and $f \in C_\loc^s(\ol{\Omega^0})$. Here, $\Delta$ denotes a forward finite difference. Moreover, when $s>3$, we have
    \[
        \norm[\big]{\partial_{\bm{u}}g[\bm{u}_1,f]-\partial_{\bm{u}}g[\bm{u}_2,f]}_{\mc{L}(\mo{\mb{X}}_+^{s-1},C_\loc^{s-2}(\ol{\Omega^0}))} \lesssim_{\norm{\bm{u}_1}_{s-1},\norm{\bm{u}_2}_s,\delta}\norm{f}_s \norm{\bm{u}_1-\bm{u}_2}_{s-1}
    \]
    when $\bm{u}_1 \in \mc{U}_\delta^{s-1}$, $\bm{u}_2 \in \mc{U}_\delta^s$, $f \in C_\loc^s(\ol{\Omega^0})$.
\end{proposition}
\begin{proof}
    Use of~\eqref{eq:g_f_difference} and \eqref{eq:g_u_difference} yields
    \[
        \Delta_{(\bm{v},0)}g[\bm{u},f]- \partial_{\bm{u}}g[\bm{u},f]\bm{v} = -g\brak[\big]{\bm{u},\bm{v}\cdot\nabla\parn[\big]{\Delta_{(\bm{v},0)}g[\bm{u},f]}},
    \]
    whence
    \begin{align*}
        \norm[\big]{\Delta_{(\bm{v},0)}g[\bm{u},f] - \partial_{\bm{u}}g[\bm{u},f]\bm{v}}_{s-2} & \lesssim_{\norm{\bm{u}}_{s-2},\delta} \norm{\bm{v}}_{s-2} \norm{\Delta_{(\bm{v},0)}g[\bm{u},f]}_{s-1} \\
                                                                                               & \lesssim_{\norm{\bm{u}}_{s},\norm{\bm{v}}_s, \delta} \norm{\bm{v}}_{s-1}^2 \norm{f}_s
    \end{align*}
    by \Cref{prop:g_lipschitz_continuity}. That is,~\eqref{eq:g_differentiability_u} holds.

    For the Lipschitz continuity, we note that
    \[
        \parn[\big]{\partial_{\bm{u}}g[\bm{u}_1,f]-\partial_{\bm{u}}g[\bm{u}_2,f]}\bm{v}=
        \begin{multlined}[t]
            g\brak[\big]{\bm{u}_2,\bm{v} \cdot \nabla g[\bm{u}_2,f]} - g\brak[\big]{\bm{u}_1,\bm{v} \cdot \nabla g[\bm{u}_2,f]}\\
            +g\brak[\big]{\bm{u}_1,\bm{v}\cdot \nabla \parn[\big]{g[\bm{u}_2,f]-g[\bm{u}_1,f]}},
        \end{multlined}
    \]
    from which the result is obtained by using the same proposition again.
\end{proof}

A straightforward induction argument from~\eqref{eq:g_derivative_u} shows that the formal $k$th derivative of $g$ with respect to $\bm{u}$ can be introduced recursively through
\begin{equation}
    \label{eq:g_kth_derivative_u}
    \partial_{\bm{u}}^k g[\bm{u},f](\bm{v}_i)_{i=1}^k \ceq - g\brak[\bigg]{\bm{u},\sum_{j=1}^k \bm{v}_j \cdot\nabla\parn[\Big]{\partial_{\bm{u}}^{k-1} g[\bm{u},f](\bm{v}_i)_{\substack{i =1 \\i \neq j}}^k}},
\end{equation}
which is well-defined when $s>k+1$, $\bm{u} \in \mc{U}_\delta^s$, $f \in C_\loc^s(\ol{\Omega^0})$, and $\bm{v}_1,\ldots, \bm{v}_k \in \mo{\mb{X}}_+^{s-1}$.
We furthermore note that
\[
    \partial_f\partial_{\bm{u}}^k g[\bm{u},f]h \ceq \partial_{\bm{u}}^k g[\bm{u}, h]
\]
by linearity, while
\[
    \partial_f^2\partial_{\bm{u}}^k g[\bm{u},f]=0.
\]

We have the following higher-order analogue of \Cref{prop:g_first_derivative_u}.

\begin{proposition}[Higher derivatives]
    \label{prop:g_second_derivative_u}
    We have
    \[
        \norm{\partial_{\bm{u}}^k g[\bm{u},f]}_{\mc{L}^k(\mo{\mb{X}}_+^{s-1},C_\loc^{s-k}(\ol{\Omega^0}))} \lesssim_{\norm{\bm{u}}_s,\delta} \norm{f}_s
    \]
    for the operator defined in~\eqref{eq:g_kth_derivative_u}, and if $s>k+2$, then
    \[
        \norm{\Delta_{(\bm{v},0)}\partial_{\bm{u}}^{k-1} g[\bm{u},f]-\partial_{\bm{u}}^k g[\bm{u},f]\bm{v}}_{\mc{L}^{k-1}(\mo{\mb{X}}_+^{s-1},C_\loc^{s-(k+1)}(\ol{\Omega^0}))}\lesssim_{\norm{\bm{u}}_s,\norm{\bm{v}}_s,\delta} \norm{f}_s \norm{\bm{v}}_{s-1}^{2}
    \]
    for all $\bm{u},\bm{u}+\bm{v} \in \mc{U}_\delta^s$ and $f \in C_\loc^s(\ol{\Omega^0})$. Moreover, when $s>k+2$ we have
    \[
        \norm{\partial_{\bm{u}}^k g[\bm{u}_1,f]-\partial_{\bm{u}}^k g[\bm{u}_2,f]}_{\mc{L}^k(\mo{\mb{X}}_+^{s-1},C_\loc^{s-(k+1)}(\ol{\Omega^0}))}\lesssim_{\norm{\bm{u}_1}_{s-1},\norm{\bm{u}_2}_s,\delta}\norm{f}_s \norm{\bm{u}_1-\bm{u}_2}_{s-1}
    \]
    when $\bm{u}_1 \in \mc{U}_\delta^{s-1}$, $\bm{u}_2 \in \mc{U}_\delta^s$ and $f \in C_\loc^s(\ol{\Omega^0})$.
\end{proposition}
\begin{proof}
    The results follow by using~\eqref{eq:g_kth_derivative_u}, together with the identities
    \begin{multline*}
        \parn[\big]{\Delta_{(\bm{v},0)}\partial_{\bm{u}}^{k-1} g[\bm{u},f]-\partial_{\bm{u}}^k g[\bm{u},f]\bm{v}}(\bm{v}_i)_{i=1}^{k-1}\\
        =-g\brak[\bigg]{\bm{u},\sum_{j=1}^{k-1} \bm{v}_j \cdot \nabla\parn[\Big]{\parn[\big]{\Delta_{(\bm{v},0)}\partial_{\bm{u}}^{k-2} g[\bm{u}, f] -\partial_{\bm{u}}^{k-1} g[\bm{u},f]\bm{v}}(\bm{v}_i)_{\substack{i=1\\i \neq j}}^{k-1}}\\
            +\bm{v} \cdot \nabla \parn[\big]{\Delta_{(\bm{v},0)}\partial_{\bm{u}}^{k-1} g[\bm{u},f](\bm{v}_i)_{i=1}^{k-1}}},
    \end{multline*}
    and
    \begin{multline*}
        \parn[\big]{\partial_{\bm{u}}^k[\bm{u}_1,f]-\partial_{\bm{u}}^kg[\bm{u}_2,f]}(\bm{v}_i)_{i=1}^k\\
        = - g\brak[\bigg]{\bm{u}_1,\sum_{j=1}^k \bm{v}_j \cdot \nabla \parn[\Big]{\parn[\big]{\partial_{\bm{u}}^{k-1}g[\bm{u}_1,f]- \partial_{\bm{u}}^{k-1}g[\bm{u}_2,f]}(\bm{v}_i)_{\substack{i=1\\i\neq j}}^k}\\
        +(\bm{u}_1-\bm{u}_2)\cdot \nabla \parn[\big]{\partial_{\bm{u}}^k g[\bm{u}_2,f](\bm{v}_i)_{i=1}^k}}
    \end{multline*}
    in an induction argument.
\end{proof}

Using the maps $g$ and $G$ provided by \Cref{prop:transport_equation_existence_uniqueness}, we may now define
\begin{align*}
    \tau[\bm{u},\eta] & \ceq g[\bm{u},\rho]
    \shortintertext{and}
    q[\bm{u},\eta]    & \ceq G[\bm{u},\rho]
\end{align*}
for $\bm{u} \in \mc{U}_\delta^s$ and $\eta \in \mc{V}_\epsilon^{s+1}$, recalling~\eqref{eq:flattened_time} and \eqref{eq:rho_definition}. These are functions in $C^s_\loc\parn[\big]{\ol{\Omega^0}}$, enjoying extra properties described in \Cref{prop:transport_equation_existence_uniqueness}. In particular, we have that $\nabla q[\bm{u},\eta]\times \nabla \tau[\bm{u},\eta]\in \mb{X}^{s-1}_-$. Moreover, the next lemma explicitly describes the formal partial derivatives of $\tau$ at the trivial solutions, which we recall were given in~\eqref{eq:trivial_solutions}.

\begin{lemma}
    We have
    \begin{equation}
        \label{eq:tau_derivatives}
        \begin{aligned}
            \partial_\bm{u}\tau[c\bm{e}_1,0]\dot{\bm{u}} & = -\frac{1}{c^2}\int_0^{x_1}\dot{u}_1(t,x_2,x_3)\,dt \\
            \partial_\eta\tau[c\bm{e}_1,0]\dot{\eta}     & = \frac{1}{cd}\int_0^{x_1}\dot{\eta}(t,x_2)\,dt
        \end{aligned}
    \end{equation}
    for every $c>0$. A particular consequence is that
    \begin{equation}
        \label{eq:q_derivative_identity}
        c\partial_\bm{u} q[c\bm{e}_1,0]D\bmf{u}[0] + \partial_\eta q[c\bm{e}_1,0] =0
    \end{equation}
    for all $c>0$.
\end{lemma}
\begin{proof}
    The derivatives given in~\eqref{eq:tau_derivatives} follow directly from~\eqref{eq:g_derivative_f}, \eqref{eq:g_derivative_u} and \eqref{eq:transport_solution}, since $\bm{\Phi}_{t,\mo{t}}= \id_{\R \times (-d,0)}$
    for all $t,\mo{t}$ at the trivial solutions. For~\eqref{eq:q_derivative_identity}, we simply compute that
    \[
        c\partial_\bm{u} q[c\bm{e}_1,0]D\bmf{u}[0]\dot{\eta} = - \frac{1}{c}\int_0^{\lambda_1} \parn[\Big]{\frac{1}{d}\dot{\eta} + \partial_1^2 \chi[\dot{\eta}]}(t,x_2,x_3)\,dt = -\frac{1}{cd}\int_0^{\lambda_1} \dot{\eta}(t,x_2)\,dt
    \]
    by~\eqref{eq:tau_derivatives} and \Cref{prop:bmf_derivative}.
\end{proof}

\subsection{Fixed point argument}
\label{sec:contraction}
Choose $\epsilon$ small enough for the last part of \Cref{prop:div_curl_problem} to hold, and such that
\[
    \mf{u}_1[\eta] >\frac{1}{2} \quad \text{for all $\eta \in \mc{V}_\epsilon^{s+1}$},
\]
which is possible since $\bmf{u}[0] = \bm{e}_1$. Then, in particular,
\[
    (c \bmf{u}[\eta]+\bms{u})_1 > \frac{4\delta}{2}-\delta= \delta
\]
for every $c > 4\delta$, $\eta \in \mc{V}_\epsilon^{s+1}$ and
\[
    \bms{u} \in \mc{B}_\delta^s \ceq \ol{B}_\delta(0;\mo{\mb{X}}_+^s),
\]
ensuring that
\[
    c\bmf{u}[\eta]+\bms{u} \in \mc{U}_\delta^s
\]
for all such $\bms{u}$, $\eta$, and $c$.

If, in addition,
\[
    c \in \mc{I} \ceq \parn[\big]{4\delta,4\delta + \epsilon^{-1}},
\]
we also have that $\norm{c\bmf{u}[\eta]+\bms{u}}_s \lesssim 1$, with the constant depending on $\epsilon$ and $\delta$. Using this, we can define a map $\map{T}{\mc{B}_\delta^s \times \mc{V}_\epsilon^{s+1} \times \mc{I}}{\mo{\mb{X}}_+^s}$ through
\[
    T[\bms{u},\eta,c] \ceq S[\eta]\parn[\Big]{\nabla \parn[\big]{h\parn[\big]{q\brak[\big]{c\bmf{u}[\eta]+\bms{u},\eta}}}\times \nabla \tau\brak[\big]{c\bmf{u}[\eta]+\bms{u},\eta}},
\]
for which we will find a fixed-point $\bms{u}[\eta,c]$ when
\[
    h \in \mc{H}_\beta^s \ceq \ol{B}_\beta(0;C^s(\R))
\]
for sufficiently small $\beta > 0$.

The differentiability properties of $T$ are given in the next lemma, where it is tacitly understood that any bound $\lesssim$ involves a potential shrinkage of $\epsilon > 0$ for the constant to be finite. For the remainder of the article, we fix a non-integer $\mf{s}>4$, and think of $h$ as fixed.

\begin{lemma}
    \label{lemma:T_deriv}
    There exist formal derivatives
    \[
        \map{D^k T}{\mc{B}^\mf{s}_\delta \times \mc{V}^{\mf{s}+1}_\epsilon \times \mc{I}}{\mc{L}^k\parn[\big]{\mo{\mb{X}}_+^{\mf{s}-1} \times \mb{Y}^\mf{s}_+\times \R,\mo{\mb{X}}_+^{\mf{s}-k}}}
    \]
    for $1 \leq k < \mf{s}-2$ such that, for $k \geq 0$,
    \begin{align*}
        \norm{D^k T[\bm{\gamma}]}_{\mc{L}^k\parn[\big]{\mo{\mb{X}}_+^{s-1}\times \mb{Y}^s_+\times \R, \mo{\mb{X}}_+^{s-k}}}                                                                        & \lesssim \beta                                                                          &  & k + 1 < s \leq \mf{s}, \\
        \norm{\Delta_{\dot{\bm{\gamma}}}D^kT[\bm{\gamma}]}_{\mc{L}^k\parn[\big]{\mo{\mb{X}}_+^{s-1}\times \mb{Y}^s_+\times \R, \mo{\mb{X}}_+^{s-(k+1)}}}                                           & \lesssim \beta \norm{\dot{\bm{\gamma}}}_{\mo{\mb{X}}_+^{s-1}\times \mb{Y}^s_+\times \R} &  & k + 2 < s \leq \mf{s},
        \shortintertext{and}
        \norm{\Delta_{\dot{\bm{\gamma}}}D^k T[\bm{\gamma}]-D^{k+1}T[\bm{\gamma}]\dot{\bm{\gamma}}}_{\mc{L}^{k}\parn[\big]{\mo{\mb{X}}_+^{s-1}\times \mb{Y}^s_+\times \R, \mo{\mb{X}}_+^{s-(k+2)}}} & = o(\norm{\dot{\bm{\gamma}}}_{\mo{\mb{X}}_+^{s-1}\times \mb{Y}^s_+\times \R})           &  & k + 3 < s \leq \mf{s}
    \end{align*}
    for all $\bm{\gamma}, \bm{\gamma} + \dot{\bm{\gamma}} \in \mc{B}^\mf{s}_\delta \times \mc{V}^{\mf{s}+1}_\epsilon \times \mc{I}$. Moreover,
    \begin{equation}
        \label{eq:T_derivatives}
        \begin{aligned}
            \partial_\bms{u} T[0,0,c]\dot{\bms{u}} & = \frac{\bm{e}_1}{c}h'\parn[\Big]{\frac{\lambda_1}{c}}\parn[\Big]{\partial_\bm{u}q[c\bm{e}_1,0]\dot{\bms{u}} + \frac{\lambda_1}{c^2}\fint_{\Omega_0^0} \dot{\ms{u}}_1\,d\bm{x}}, \\
            \partial_\eta T[0,0,c]                 & = 0
        \end{aligned}
    \end{equation}
    for every $c \in \mc{I}$.
\end{lemma}
\begin{proof}
    $T$ is a composition of the maps given above the lemma. Repeated use of \Cref{Cor:composition} gives the regularity (see also \Cref{Rem:composition}). For~\eqref{eq:T_derivatives}, start by observing that
    \begin{align*}
        \partial_\bms{u} T[0,0,c]\dot{\bms{u}} & = S[0]\parn[\bigg]{\nabla\parn[\Big]{h'\parn[\Big]{\frac{\lambda_1}{c}}\partial_\bm{u}q[c\bm{e}_1,0]\dot{\bms{u}}} \times \frac{\bm{e}_1}{c}}  \\
                                               & = \frac{1}{c}h'\parn[\Big]{\frac{\lambda_1}{c}}S[0]\parn[\Big]{\nabla \times \parn[\big]{\bm{e}_1\partial_\bm{u}q[c\bm{e}_1,0]\dot{\bms{u}}}},
    \end{align*}
    which yields the first derivative after using~\eqref{eq:tau_derivatives}. Similarly,
    \[
        \partial_\eta T[0,0,c]\dot{\eta} = \frac{1}{c}h'\parn[\Big]{\frac{\lambda_1}{c}}S[0]\parn[\Big]{\nabla \times \parn[\big]{\bm{e}_1\parn{\partial_\eta q[c\bm{e}_1,0] + c\partial_\bm{u}q[c\bm{e}_1,0]D\bmf{u}[0]}\dot{\eta}}},
    \]
    which vanishes identically by~\eqref{eq:q_derivative_identity}.
\end{proof}

By \Cref{lemma:T_deriv}, we may choose $\beta$ small enough for
\[
    \norm{\Delta_{\dot{\bm{\gamma}}} T[\bm{\gamma}]}_{\mo{\mb{X}}_+^{\mf{s}-1}} \leq \frac{1}{2}\norm{\dot{\bm{\gamma}}}_{\mo{\mb{X}}_+^{\mf{s}-1}\times \mb{Y}^\mf{s}_+\times \R}
\]
to hold for all $\bm{\gamma}, \bm{\gamma} + \dot{\bm{\gamma}} \in \mc{B}^\mf{s}_\delta \times \mc{V}^{\mf{s}+1}_\epsilon \times \mc{I}$. In particular, $T[\wildcard,\eta,c]$ has a unique fixed-point
\[
    \bms{u}[\eta,c] \in \mc{B}_\delta^\mf{s}
\]
for every $(\eta,c) \in \mc{V}_\epsilon^{\mf{s}+1} \times \mc{I}$, furnished by Banach's fixed-point theorem. This is because the ball $\mc{B}_\delta^\mf{s}$ is closed with respect to $\norm{\wildcard}_{\mf{s}-1}$.

The next goal is to establish regularity properties of this fixed point, in order to facilitate the eventual Lyapunov--Schmidt reduction. We expect the loss of derivatives with respect to $\eta$ and $c$ to be inherited from the map $T$, so this regularity must necessarily be measured with respect to weaker norms.

\begin{proposition}
    \label{prop:u'}
    The fixed point has formal derivatives
    \[
        \map{D^k\bms{u}}{\mc{V}^{\mf{s}+1}_\epsilon\times \mc{I}}{\mc{L}^k\parn[\big]{\mb{Y}^\mf{s}_+\times \R,\mo{\mb{X}}_+^{\mf{s}-k}}}
    \]
    for $1 \leq k < \mf{s} -2$ such that, for $k \geq 0$,
    \begin{align*}
        \norm{D^k \bms{u}[\eta,c]}_{\mc{L}^k\parn[\big]{\mb{Y}^{s}_+\times \R,\mo{\mb{X}}_+^{s-k}}}                                                                           & \lesssim \beta                                                  &  & k+1 < s \leq \mf{s}, \\
        \norm{\Delta_{(\dot{\eta},\dot{c})} D^k \bms{u}[\eta,c]}_{\mc{L}^k\parn[\big]{\mb{Y}^s_+\times \R,\mo{\mb{X}}_+^{s-(k+1)}}}                                           & \lesssim \beta \parn[\big]{\norm{\dot{\eta}}_s + \abs{\dot{c}}} &  & k+2 < s \leq \mf{s}, \\
        \shortintertext{and}
        \norm{\Delta_{(\dot{\eta},\dot{c})}D^k \bms{u}[\eta,c]-D^{k+1}\bms{u}[\eta,c](\dot{\eta},\dot{c})}_{\mc{L}^k\parn[\big]{\mb{Y}^s_+\times \R,\mo{\mb{X}}_+^{s-(k+2)}}} & =o\parn[\big]{\norm{\dot{\eta}}_s + \abs{\dot{c}}}              &  & k+3 < s \leq \mf{s}
    \end{align*}
    for all $(\eta,c),(\eta,c)+(\dot{\eta},\dot{c}) \in \mc{V}_\epsilon^{\mf{s}+1} \times \mc{I}$.
\end{proposition}
\begin{proof}
    This follows by applying \Cref{Cor:fixedpoint}.
\end{proof}

\section{Main result}
\label{sec:Proof of main result}
\subsection{The dynamic boundary condition}
\label{sec:Dynamic BC}

For any fixed $(\eta,c) \in \mc{V}_\epsilon^{\mf{s}+1} \times \mc{I}$, we have constructed the solution
\[
    \bm{u}[\eta,c] = c \bmf{u}[\eta] + \bms{u}[\eta,c]
\]
to~\eqref{eq:flattened_euler}--\eqref{eq:flattened_kinematic}, where $\bmf{u}[0]=\bm{e}_1$, and $\bms{u}[0,c]=0$ for all $c \in \mc{I}$. By now substituting this solution into the dynamic boundary condition~\eqref{eq:dynamic_bdry_flattened}, we obtain the equation
\begin{equation}
    \label{eq:problem_reduced_to_bdry}
    F[\eta,c]+R[\eta,c]=0
\end{equation}
where $F \colon \mc{V}^{\mf{s}+1}_\epsilon\times \mc{I}\to \mb{Y}_+^{\mf{s}-1}$ and $R \colon \mc{V}^{\mf{s}+1}_\epsilon\times \mc{I}\to \mb{Y}_+^s$ are defined by
\begin{align*}
    F[\eta,c] & \ceq \frac{c^2}{2\rho^2}\parn[\big]{\abs{\bmf{u}[\eta]}^2 + (\nabla \eta \cdot \bmf{u}[\eta])^2} + g\eta + \sigma \nabla \cdot \bm{n}[\eta]-Q(c) \\
    \shortintertext{and}
    R[\eta,c] & \ceq
    \begin{multlined}[t]
        \frac{1}{2\rho^2}\parn[\big]{\abs{\bms{u}[\eta,c]}^2 + (\nabla\eta \cdot \bms{u}[\eta])^2}\\
        +\frac{c}{\rho^2}\parn[\big]{\bmf{u}[\eta] \cdot \bms{u}[\eta,c] + (\nabla \eta \cdot \bmf{u}[\eta])(\nabla \eta \cdot \bms{u}[\eta,c])}-h\parn[\big]{q\brak[\big]{c \bmf{u}[\eta] + \bms{u}[\eta,c],\eta}},
    \end{multlined}
\end{align*}
respectively. It is important to note that we have both suppressed the restriction to $S^0$, and tacitly dropped the final (zero) component of $\bms{u}$ and $\bmf{u}$ in all but the final term involving $q$.

While $F$ is smooth, the map $R$ is not differentiable in the usual sense, since neither $q$ nor $\bms{u}$ is differentiable. However, as we have come to expect, it does inherit some regularity properties with respect to a weaker norm from $q$ and $\bms{u}$. These are summarised in the lemma below.

\begin{lemma}
    \label{lemma:R_properties}
    There are formal derivatives
    \[
        \map{D^k R}{\mc{V}^{\mf{s}+1}_\epsilon\times \mc{I}}{\mc{L}^k\parn[\big]{\mb{Y}^{\mf{s}}_+\times \R,\mb{Y}_+^{\mf{s}-k}}}
    \]
    for $1 \leq k < \mf{s} -2$ such that, for $k \geq 0$,
    \begin{align*}
        \norm{D^k R[\eta,c]}_{\mc{L}^k\parn[\big]{\mb{Y}^s_+\times \R,\mb{Y}_+^{s-k}}}                                                                       & \lesssim \beta                                                  &  & k+1 < s \leq \mf{s}, \\
        \norm{\Delta_{(\dot{\eta},\dot{c})} D^k R[\eta,c]}_{\mc{L}^k\parn[\big]{\mb{Y}^s_+\times \R,\mb{Y}_+^{s-(k+1)}}}                                     & \lesssim \beta \parn[\big]{\norm{\dot{\eta}}_s + \abs{\dot{c}}} &  & k+2 < s \leq \mf{s}, \\
        \shortintertext{and}
        \norm{\Delta_{(\dot{\eta},\dot{c})}D^k R[\eta,c]-D^{k+1}R[\eta,c](\dot{\eta},\dot{c})}_{\mc{L}^k\parn[\big]{\mb{Y}^s_+\times \R,\mb{Y}_+^{s-(k+2)}}} & =o\parn[\big]{\norm{\dot{\eta}}_s + \abs{\dot{c}}}              &  & k+3 < s \leq \mf{s}
    \end{align*}
    for all $(\eta,c),(\eta,c)+(\dot{\eta},\dot{c}) \in \mc{V}_\epsilon^{\mf{s}+1} \times \mc{I}$. Moreover,
    \[
        \partial_\eta R[0,c] = 0
    \]
    for all $c \in \mc{I}$.
\end{lemma}
\begin{proof}
    We use the fact that $R$ can be seen as a composition of several maps, for which we already know the regularity. Applying \Cref{Cor:composition} repeatedly gives the first part of the result, except for the factor $\beta$ in the norm and Lipschitz estimates. For the last term, it is evident that we have this factor $\beta$. For the remaining terms, we get a factor $\beta$ from \Cref{prop:u'}.

    To prove the last part of the lemma, we note that
    \begin{align*}
        \partial_\eta R[0,c]\dot{\eta} & = c \bm{e}_1 \cdot \partial_\eta\bms{u}[0,c]\dot{\eta}-h'\parn[\Big]{\frac{\lambda_1}{c}}\parn[\big]{\partial_\bm{u}q[c\bm{e}_1,0](\partial_\eta \bms{u}[0,c] + cD\bmf{u}[0])+\partial_\eta q[c\bm{e}_1,0]}\dot{\eta} \\
                                       & = c\bm{e}_1 \cdot \parn[\big]{\parn{\partial_\eta \bms{u}[0,c] - \partial_\bms{u}T[0,0,c]\partial_\eta \bms{u}[0,c]}\dot{\eta}}
    \end{align*}
    by~\eqref{eq:q_derivative_identity} and \eqref{eq:T_derivatives}. This expression vanishes because
    \[
        \partial_\eta \bms{u}[0,c] - \partial_\bms{u}T[0,0,c]\partial_\eta \bms{u}[0,c] = \partial_\eta T[0,0,c] = 0
    \]
    by virtue of $\bms{u}$ being a fixed point, and the fact that $\partial_\eta T[0,0,c] \equiv 0$.
\end{proof}

For the smooth part $F$ in~\eqref{eq:problem_reduced_to_bdry}, we have the derivative
\[
    \partial_\eta F[0,c]\dot{\eta}=g\dot{\eta}-\sigma\Delta\dot{\eta}+c^2\partial_1^2 \chi[\dot{\eta}],
\]
where $\chi$ is the potential from \Cref{prop:bmf_derivative}. Let
\[
    L[c]\ceq \partial_\eta F[0,c] \colon\mb{Y}^{\mf{s}+1}_+ \to \mb{Y}^{\mf{s}-1}_+,
\]
which we can extend to an operator $\mf{L}[c] \colon \mb{Y}^{\mf{s}+1}\to \mb{Y}^{\mf{s}-1}$. Here $\mb{Y}^\mf{s}$ denotes the space of all $\Lambda$-periodic $C^\mf{s}(\R^2)$ functions, with no symmetry requirement. The dispersion relation is then obtained by considering the equation
\begin{equation}
    \label{eq:F_deriv_1d_nullspace}
    \mf{L}[c]e^{i\bm{k}\cdot\bm{x}'}=0
\end{equation}
for $\bm{k} \in \Lambda^*$.
We are now ready to prove \Cref{prop:F_kernel}, concerning the number of solutions to this equation.
\begin{proof}[Proof of \Cref{prop:F_kernel}]
    Using the definition of $L[c]$, we see that
    \begin{equation}
        \label{eq:DF_on_fourier_mode}
        \mf{L}[c] e^{i\bm{k}\cdot\bm{x}'}=\ell_{\bm{k}}(c)e^{i\bm{k}\cdot\bm{x}'},
    \end{equation}
    with $\ell_{\bm{0}}(c) \ceq g$. Clearly, $\bm{\kappa}$ is a solution to~\eqref{eq:F_deriv_1d_nullspace} if and only if $c = c^*(\sigma)$, and $\bm{k}$ cannot be a solution if $k_1 = 0$. Thus, we may restrict our attention to the equivalent equation
    \[
        \sigma \parn[\big]{\abs{\bm{k}}^2\varphi_\bm{k} - \abs{\bm{\kappa}}^2\varphi_\bm{\kappa}} + g\parn[\big]{\varphi_\bm{k} - \varphi_\bm{\kappa}} = 0, \qquad \varphi_\bm{k} \ceq \frac{\abs{\bm{k}}^2}{k_1^2 \mo{\mf{f}}_{\abs{\bm{k}}}} =\frac{\abs{\bm{k}}\tanh(\abs{\bm{k}}d)}{k_1^2}
    \]
    for $\bm{k}$ with $k_1 \neq 0$.

    There are now two possibilities if $\bm{k}$ is a solution: Either
    \[
        \abs{\bm{k}}^2\varphi_\bm{k} - \abs{\bm{\kappa}}^2\varphi_\bm{\kappa} = \varphi_\bm{k} - \varphi_\bm{\kappa} = 0,
    \]
    in which case $\abs{\bm{k}} = \abs{\bm{\kappa}}$, and therefore $\bm{k} \in \brac{\pm\bm{\kappa},\pm \bm{\kappa}^*}$; or both coefficients are nonzero, and
    \[
        \sigma = \sigma^*(\bm{k}) \ceq -g\frac{\varphi_\bm{k}-\varphi_\bm{\kappa}}{\abs{\bm{k}}^2\varphi_\bm{k} - \abs{\bm{\kappa}}^2\varphi_\bm{\kappa}}.
    \]
    The result now follows by letting $\Sigma$ be the complement of the range of $\sigma^*$ over all applicable $\bm{k}$.
\end{proof}

From here on, we will assume that $\sigma \in \Sigma$, where $\Sigma$ is as described in \Cref{prop:F_kernel}, and write $c^*$ instead of $c^*(\sigma)$ for simplicity.

\begin{lemma}
    \label{lemma:F_bifurcation_properties}
    The operator $L[c^*]$ is Fredholm of index $0$, with
    \[
        \ker{L}[c^*] =\lspan{\brac{\ac{1}{\eta}}} \qquad \text{and} \qquad \mb{Y}^{\mf{s}-1}_+ = \ran{L}[c^*]\oplus\lspan{\brac{\ac{1}{\eta}}},
    \]
    where $\ac{1}{\eta}$ is like in \Cref{thm:main informal}. Moreover,
    \begin{equation}
        \label{eq:F_mixed_derivative}
        \partial_c L[c^*]\ac{1}{\eta}=-2c^* \kappa_1^2 \mo{\mf{f}}_{\abs{\bm{\kappa}}}\ac{1}{\eta}.
    \end{equation}
\end{lemma}
\begin{proof}
    Let
    \[
        \dot{\eta}(\bm{x}')=\sum_{\bm{k} \in \Lambda^*}\dot{\eta}_{\bm{k}}e^{i\bm{k} \cdot \bm{x}'}
    \]
    be an element in the kernel. From \Cref{prop:F_kernel}, we know that only $\dot{\eta}_\bm{k}$ with $\abs{\bm{k}}=\abs{\bm{\kappa}}$ can be nonzero, and by $(+)$-symmetry these coefficients must also be equal. Thus
    \[
        \dot{\eta}= \dot{\eta}_\bm{\kappa}\parn[\big]{e^{i\bm{\kappa} \cdot \bm{x}'} + e^{-i\bm{\kappa} \cdot \bm{x}'} + e^{i\bm{\kappa}^* \cdot \bm{x}'} + e^{-i\bm{\kappa}^* \cdot \bm{x}'}} = 4\dot{\eta}_\bm{\kappa}\ac{1}{\eta},
    \]
    showing that $\ker{L}[c^*]=\lspan{\brac{\ac{1}{\eta}}}$. From~\eqref{eq:DF_on_fourier_mode}, we also see that the cokernel of $L[c^*]$ is spanned by the same Fourier modes. Hence we get that $L[c^*]$ is a Fredholm operator of index 0.

    For the final part, we note that the derivative is given by
    \[
        \partial_c L[c^*]\dot{\eta}=2c^*\partial_1^2 \chi[\dot{\eta}]
    \]
    in general, which yields~\eqref{eq:F_mixed_derivative} after evaluating at $\dot{\eta} = \ac{1}{\eta}$.
\end{proof}
\subsection{Lyapunov--Schmidt reduction}
We are now ready to perform a variant of the usual Lyapunov--Schmidt reduction of~\eqref{eq:problem_reduced_to_bdry}, working under the same assumptions and notation of \Cref{lemma:F_bifurcation_properties}. We also define the orthogonal projections
\[
    \mc{P} \text{ onto } \ker L[c^*], \qquad \text{and} \qquad \mc{Q}=\id-\mc{P} \text{ onto } (\ker L[c^*])^\perp
\]
in $L^2(\R^2 /\Lambda)$, which restrict to projections in the $\mb{Y}_+$-spaces. In particular
\[
    \mc{P}\mb{Y}^{\mf{s}+1}_{+}=\ker{L}[c^*],
    \qquad\text{and}\qquad
    \mc{Q}\mb{Y}^{\mf{s}-1}_+=\ran{L}[c^*],
\]
allowing us to express~\eqref{eq:problem_reduced_to_bdry} as the system
\begin{align}
    \mc{Q}F[t\ac{1}{\eta}+\ac{\perp}{\eta},c]+\mc{Q}R[t\ac{1}{\eta}+\ac{\perp}{\eta},c] & =0,
    \label{eq:inf_dim_prob}                                                                   \\
    \mc{P}F[t\ac{1}{\eta}+\ac{\perp}{\eta},c]+\mc{P}R[t\ac{1}{\eta}+\ac{\perp}{\eta},c] & =0,
    \label{eq:finite_dim_prob}
\end{align}
if we write
\[
    \mc{P}\eta \eqc t\ac{1}{\eta} \qquad \text{and} \qquad \mc{Q}\eta \eqc \ac{\perp}{\eta}.
\]

Since $L[c^*]$ is invertible as an operator from $\mc{Q}\mb{Y}^{\mf{s}+1}_+$ to $\mc{Q}\mb{Y}^{\mf{s}-1}_+$, we can apply $-L[c^*]^{-1}$ to~\eqref{eq:inf_dim_prob} and get the equivalent equation
\[
    \mc{M}[\ac{\perp}{\eta},t,c] \ceq \ac{\perp}{\eta}-L[c^*]^{-1}(\mc{Q}F[t\ac{1}{\eta}+\ac{\perp}{\eta},c]+\mc{Q}R[t\ac{1}{\eta}+\ac{\perp}{\eta},c])=\ac{\perp}{\eta}
\]
after adding $\ac{\perp}{\eta}$ to both sides. If we consider $\mc{M}$ to be a map $U \times \ol{\mc{K}} \to \ol{\mc{K}}$ for sufficiently small neighbourhoods $U = V \times W$ of $(0,c^*) \in \R^2$ and $\mc{K}$ of $0 \in \mc{Q}\mb{Y}^{\mf{s}+1}_+$, then we can ensure that $\mc{M}[\wildcard,t,c]$ has a unique fixed point $\ac{\perp}{\eta}[t,c] \in \ol{\mc{K}}$; after possibly making $\beta$ smaller. This is because
\[
    \partial_{\ac{\perp}{\eta}} \mc{M}[0,0,c^*] = 0
\]
in the formal sense, combined with an argument like the one for $\bms{u}$ just prior to \Cref{prop:u'}. Crucially, we need the estimates from \Cref{lemma:R_properties} to deal with the (traditionally) non-differentiable part.

\begin{lemma}
    \label{lemma:eta_perp_differentiability}
    The uniformly bounded map $\map{\ac{\perp}{\eta}}{U}{\mc{Q}Y_+^{\mf{s}+1}}$ has formal derivatives
    \[
        \map{D^k\ac{\perp}{\eta}}{U}{\mc{L}^k\parn[\big]{\R^2,\mc{Q}\mb{Y}_+^{\mf{s}+1-(k-1)}}}
    \]
    for $1 \leq k < \mf{s} -2$ such that
    \begin{align*}
        \norm{D^k\ac{\perp}{\eta}[t,c]}_{\mc{L}^k\parn[\big]{\R^2,\mc{Q}\mb{Y}_+^{s+1-(k-1)}}}                                                                         & \lesssim 1                           &  & k+1 < s \leq \mf{s},
        \shortintertext{for $k \geq 1$, and}
        \norm{\Delta_{(\dot{t},\dot{c})}D^k\ac{\perp}{\eta}[t,c]}_{\mc{L}^k\parn[\big]{\R^2,\mc{Q}\mb{Y}_+^{s+1-k}}}                                                   & \lesssim \abs{\dot{t}}+\abs{\dot{c}} &  & k+2 < s \leq \mf{s}, \\
        \norm{\Delta_{(\dot{t},\dot{c})}D^k\ac{\perp}{\eta}[t,c]-D^{k+1}\ac{\perp}{\eta}[t,c](\dot{t},\dot{c})}_{\mc{L}^k\parn[\big]{\R^2,\mc{Q}\mb{Y}_+^{s+1-(k+1)}}} & =o(\abs{\dot{t}}+\abs{\dot{c}})      &  & k+3 < s \leq \mf{s},
    \end{align*}
    for $k \geq 0$, for all $(t,c),(t,c)+(\dot{t},\dot{c}) \in U$. Moreover,
    \begin{equation}
        \label{eq:eta_perp_vanishing}
        \ac{\perp}{\eta}[0,c] = \partial_t \ac{\perp}{\eta}[0,c^*] = 0
    \end{equation}
    for all $c \in W$.
\end{lemma}
\begin{proof}
    We can apply \Cref{Prop:fixed_point_diff} to $\mc{M}$ to obtain the bounds for $k = 0,1$, while \Cref{Cor:fixedpoint} must be applied instead for higher values of $k$.

    For the last part, it is clear by definition that $\mc{M}[0,0,c] = 0$ for every $c \in W$, which immediately yields $\ac{\perp}{\eta}[0,c]=0$. Furthermore (see \Cref{Appendix:fixed_points}),
    \begin{align*}
        \partial_t \ac{\perp}{\eta}[0,c^*] & = \parn[\big]{\id - \partial_{\ac{\perp}{\eta}}\mc{M}[0,0,c^*]}^{-1}\partial_t \mc{M}[0,0,c^*] \\
                                           & =\partial_t \mc{M}[0,0,c^*]                                                                    \\
                                           & = -L[c^*]^{-1}QL[c^*]\ac{1}{\eta}
    \end{align*}
    vanishes since $\ac{1}{\eta} \in \ker{L}[c^*]$.
\end{proof}

This means that we can substitute $\ac{\perp}{\eta}[t,c]$ into~\eqref{eq:finite_dim_prob} to obtain the two-dimensional problem
\begin{equation}
    \label{eq:finite_dim_prob_eta_perp}
    K[t,c] \ceq \mc{P}F[t\ac{1}{\eta}+\ac{\perp}{\eta}[t,c],c]+\mc{P}R[t\ac{1}{\eta}+\ac{\perp}{\eta}[t,c],c]=0
\end{equation}
for $(t,c) \in U$.
By applying \Cref{Cor:composition} (see \Cref{Rem:composition}) we find that $\map{K}{U}{\ker{L}[c^*]}$ is $C^{\floor{\mf{s}}-2}$, and therefore at least twice continuously differentiable, since $\mf{s}> 4$.
Moreover, $K[0,c]=0$ for all $c \in W$ by~\eqref{eq:eta_perp_vanishing}, whence
\[
    K[t,c] = tN[t,c],
\]
where
\[
    N[t,c] \ceq \int_0^1 \partial_t K[tz,c]\,dz
\]
is $C^{\floor{\mf{s}}-3}$. In particular, this means that~\eqref{eq:finite_dim_prob_eta_perp} is equivalent to
\begin{equation}
    \label{eq:bifurcation equation}
    N[t,c]=0
\end{equation}
for nontrivial ($t \neq 0$) solutions.

We point out that
\begin{equation}
    \label{eq:H_identity}
    N[0,c] = \partial_t K[0,c] = \mc{P} L[c]\parn[\big]{\ac{1}{\eta} + \partial_t \ac{\perp}{\eta}[0,c]}
\end{equation}
for all $c \in W$, by the final part of \Cref{lemma:R_properties}. For later purposes, we also record the following result concerning the regularity of $\bm{u}$ and $\eta$, which follows by applying \Cref{Cor:composition}.
\begin{corollary}
    \label{cor:t_c_dependence_of_solutions}
    If $\map{(\bm{u},\eta)}{U}{\mo{\mb{X}}_+^\mf{s} \times\mb{Y}_+^{\mf{s}+1}}$ is defined by
    \begin{align*}
        \eta[t,c]   & \ceq t\ac{1}{\eta} + \ac{\perp}{\eta}[t,c],      \\
        \bm{u}[t,c] & \ceq c\bmf{u}[\eta[t,c]] + \bms{u}[\eta[t,c],c],
    \end{align*}
    then $(\bm{u},\eta) \in C^2\parn[\big]{U;\mo{\mb{X}}^{\mf{s}-3}_+ \times \mb{Y}_+^{\mf{s}-1}}$
\end{corollary}

\subsection{Existence result}
We are now prepared to state and prove a more precise version of the existence part of the main theorem, \Cref{thm:main informal}, excluding the rotationality of the solutions. We postpone this to \Cref{prop:nonzero_vorticity}.

\begin{theorem}
    \label{thm:Main_result}
    If $\sigma \in \Sigma$, and $V$ is as in \Cref{lemma:eta_perp_differentiability}, then there exist solutions
    \[
        (\bm{u}[t],\eta[t],c(t))\in \mo{\mb{X}}_+^\mf{s}\times \mb{Y}^{\mf{s}+1}_+\times \R
    \]
    to~\eqref{eq:flatproblem} for all $t \in V$, such that $(\bm{u}[0],\eta[0],c(0)) = (\ac{0}{\bm{u}}[c^*],0,c^*)$,
    \[
        (\bm{u},\eta,c) \in \bigcap_{n=0}^{\floor{\mf{s}}-3} C^n\parn[\big]{V; \mo{\mb{X}}_+^{\mf{s}-1-n} \times \mb{Y}^{\mf{s}+1-n}_+ \times W},
    \]
    and
    \[
        \eta[t]=t\ac{1}{\eta} + o(t)
    \]
    in $\mb{Y}_+^\mf{s}$ as $t \to 0$. Moreover, for each $t \in V$, there is a corresponding $\md{\Lambda}$-periodic, $(+)$-symmetric solution in $C^\mf{s}(\ol{\Omega^{\eta[t]}}; \R^3)\times C^{\mf{s}+1}(\R^2)\times \R$ to the original problem~\eqref{eq:kinematic}, \eqref{eq:unflattened_problem}, obtained by inverting the flattening transform.
\end{theorem}
\begin{proof}
    Differentiating~\eqref{eq:H_identity}, we obtain
    \begin{align*}
        \partial_c N (0,c^*) & = \mc{P}\partial_cL[c^*](\ac{1}{\eta}+\partial_t\ac{\perp}{\eta}[0,c^*])+\mc{P}L[c^*]\partial_c\partial_t\ac{\perp}{\eta}[0,c^*] \\
                             & =\mc{P}\partial_cL[c^*]\ac{1}{\eta},
    \end{align*}
    which is nonzero by~\eqref{eq:F_mixed_derivative}. Here we have used both~\eqref{eq:eta_perp_vanishing} from \Cref{lemma:eta_perp_differentiability}, and the fact that $\mc{P}L[c^*] = 0$ by definition of $\mc{P}$. The implicit function theorem now implies that all solutions to~\eqref{eq:bifurcation equation} are of the form $(t,c(t))$ for $t \in V$, with $c[0]=c^*$, possibly after shrinking $V$.

    Tracing back, it is clear that
    \begin{align*}
        \eta[t]   & = t\ac{1}{\eta} + \ac{\perp}{\eta}[t,c(t)],    \\
        \bm{u}[t] & = c(t)\bmf{u}[\eta[t]] + \bms{u}[\eta[t],c(t)]
    \end{align*}
    is the desired corresponding surface profile and velocity field, respectively. The regularity properties of the maps $\bm{u}$ and $\eta$ follow from \Cref{Cor:composition}, and the expansion of $\eta$ from~\eqref{eq:eta_perp_vanishing}.

    To show that $\bm{u}[t]$ and $\eta[t]$ are in fact $\md{\Lambda}$-periodic, note that $\eta[t](\wildcard + \bm{\lambda}/2)$ also gives a solution with the same projection onto the kernel. Moreover, it is $(+)$-symmetric since
    \[
        \eta[t](R_j \bm{x}' + \bm{\lambda}/2) = \eta[t](\bm{x}' + R_j \bm{\lambda}/2) = \eta[t](\bm{x}' + \bm{\lambda}/2),
    \]
    where the first equality is $(+)$-symmetry of $\eta[t]$ and the second is $\Lambda$-periodicity. We conclude by uniqueness that $\eta[t](\wildcard + \bm{\lambda}/2) = \eta[t]$, and similarly that $\eta[t](\wildcard + \bm{\lambda}^*/2) = \eta[t]$. The same is clearly true for $\bm{u}$.

    Finally, the inverse of the flattening transformation introduced in \Cref{sec:flattening} is given by composition with $\Pi^{-1}(\bm{x})=\bm{x}- \rho^{-1}\pi(\bm{x})\bm{e}_3$ and multiplication of the velocity field by $(\rho^{-1}J) \circ \Pi^{-1}$, from which the existence of a solution to~\eqref{eq:kinematic}, \eqref{eq:unflattened_problem} with the stated properties follows.
\end{proof}

\subsection{Expansion of the solution}
We may, without loss of generality, suppose that
\[
    cQ'(c) = c^2 + \ac{0}{q}[c]h'(\ac{0}{q}[c]) > 0
\]
for all $c \in W$, by possibly shrinking $\beta$. If, in addition
\[
    h'(\ac{0}{q}[c^*]) \neq 0,
\]
then we can show that \Cref{thm:Main_result} will produce solutions with nonzero vorticity when $0 < \abs{t} \ll 1$.

We will prove this fact by expanding the solution curve in terms of $t$. Instead of directly expanding everything with respect to $t$, we follow the lines of the proof. That is, we find a solution $\bm{u}[t,c],\eta[t,c]$ to problem~\eqref{eq:flatproblem}, except that the dynamic boundary condition is projected onto $(\ker L[c^*])^\perp$. Finally, we determine the expansion of $c(t)$ so that also its projection onto $\ker L[c^*]$ is satisfied.

One of the main reasons for expanding the functions in this manner, is that they are twice differentiable (in the sense of \Cref{lemma:eta_perp_differentiability}) as functions of $c$ and $t$. This means that they can be expanded to second order in terms of these variables. On the other hand, if $\mf{s}<5$, then $c$ as a function of $t$ is only $C^1$, so the full solution can only be expanded to first order in $t$. In the end this means that we obtain slightly more information by expanding in two steps. We still do not get a second order expansion of the full solution in the worst case scenario, but we get sufficient information about the second order terms to show that the solution has nonzero vorticity.

Our starting point is the expansion
\begin{equation}
    \label{eq:sol_expansion_1}
    \begin{alignedat}{9}
        \bm{u}[t,c] & = & \ac{0}{\bm{u}}[c]+\ac{1}{\bm{u}}[c]t & +\ac{2}{\bm{u}}[c]t^2 &  & +o(t^2) &  & \quad \text{in } \mo{\mb{X}}_+^{\mf{s}-3},           \\
        \eta[t,c]   & = & \ac{1}{\eta}t                        & +\ac{2}{\eta}[c]t^2   &  & +o(t^2) &  & \quad \text{in } \mb{Y}_{+}^{\mf{s}-1},              \\
        \tau[t,c]   & = & \ac{0}{\tau}[c]+\ac{1}{\tau}[c]t     & +\ac{2}{\tau}[c]t^2   &  & +o(t^2) &  & \quad \text{in } C_\loc^{\mf{s}-3}(\ol{\Omega^0}),   \\
        q[t,c]      & = & \ac{0}{q}[c]+\ac{1}{q}[c]t           & +\ac{2}{q}[c]t^2      &  & +o(t^2) &  & \quad \text{in } \mb{Y}_+^{\mf{s}-3}(\ol{\Omega^0}),
    \end{alignedat}
\end{equation}
as $t \to 0$, which exists from \Cref{thm:Main_result}, \Cref{cor:t_c_dependence_of_solutions}, \Cref{Cor:composition} and the results in \Cref{sec:transport}. Here $\ac{0}{\bm{u}}$ is as in~\eqref{eq:trivial_solutions}, $\ac{0}{\tau}$ and $\ac{0}{q}$ are as in~\eqref{eq:trivial_time_functions}, and $\ac{1}{\eta}$ is as in \Cref{thm:main informal}. While the coefficients depend on $c \in W$ in general, we will often suppress this argument for notational convenience.

Substituting these expressions into the governing equations, the terms of order $t^0$ clearly vanish by design, and identifying the terms of order $t$ and $t^2$ yields the following:
\subsubsection*{Terms of order $t$}
By using the explicit expressions for $\ac{0}{\tau}$ and $\ac{0}{q}$, we find that~\eqref{eq:flattened_euler} is
\[
    \nabla \times (\ac{1}{\bm{u}} + B[\ac{1}{\eta}]\ac{0}{\bm{u}}) = \frac{1}{c}h'(\ac{0}{q}) \nabla \times (\ac{1}{q}\bm{e}_1)
\]
at the first order, with $B$ as defined in \Cref{lemma:derivative_C}. Consequently
\begin{align*}
    \ac{1}{\bm{u}} & = -B[\ac{1}{\eta}]\ac{0}{\bm{u}} + \frac{1}{c}h'(\ac{0}{q})\ac{1}{q}\bm{e}_1 + \nabla \ac{1}{\phi} + \ac{1}{U}\bm{e}_1                                       \\
                   & =\parn[\Big]{\frac{1}{c}h'(\ac{0}{q})\ac{1}{q} + \frac{c}{d}\ac{1}{\eta} + \ac{1}{U}}\bm{e}_1 - c\partial_1 \pi[\ac{1}{\eta}]\bm{e}_3 + \nabla \ac{1}{\phi},
\end{align*}
where $\ac{1}{\phi}$ is $\Lambda$-periodic and $\ac{1}{U}$ is a constant, to be determined. There is no constant in the $\bm{e}_2$-direction by the $(+)$-symmetry of $\bm{u}$.

From~\eqref{eq:flattened_time} and \eqref{eq:flattened_time_boundary}, we see that
\[
    c\partial_1 \ac{1}{\tau} = \frac{1}{d}\ac{1}{\eta} - \frac{1}{c}\ac{1}{u}_1, \qquad \text{and} \qquad \rest[\big]{\ac{1}{\tau}}_{x_1 = 0} = 0,
\]
and therefore
\[
    c \ac{1}{\tau} = \int_0^{x_1} \parn[\Big]{\frac{1}{d}\ac{1}{\eta}-\frac{1}{c}\ac{1}{u}_1}(y_1,x_2,x_3)\,dy_1 \qquad \text{and} \qquad \ac{1}{q} = -\frac{1}{c^2}\int_0^{\lambda_1} \ac{1}{u}_1\,dx_1 = 0,
\]
where~\eqref{eq:integral_cond_flattened} was used for the final equality. With the knowledge that $\ac{1}{q} = 0$, the same equality with the expression for $\ac{1}{\bm{u}}$ inserted also yields $\ac{1}{U} = 0$.

Next,~\eqref{eq:flattened_divergence} and \eqref{eq:flattened_kinematic} impose the equations
\[
    \Delta \ac{1}{\phi} = 0 \quad \text{in }\Omega_0^0, \qquad \text{and} \qquad \partial_3\ac{1}{\phi} = c\partial_1 \pi[\ac{1}{\eta}] \quad \text{on } \partial \Omega_0^0,
\]
for $\ac{1}{\phi}$, entailing that $\ac{1}{\phi} = c \partial_1 \chi[\ac{1}{\eta}] = c \partial_1 \ac{1}{\eta}\mf{f}_{\abs{\bm{\kappa}}}$. In summary
\[
    \ac{1}{\bm{u}} = c\parn[\Big]{\frac{1}{d}\ac{1}{\eta}\bm{e}_1 - \partial_1 \pi[\ac{1}{\eta}]\bm{e}_3} + \nabla \ac{1}{\phi}
\]
leading to
\[
    \ac{1}{\tau} = - \frac{1}{c^2}\ac{1}{\phi}
\]
by using the integral formula above.

Finally, the dynamic boundary condition~\eqref{eq:dynamic_bdry_flattened} becomes
\[
    \ell_{\bm{\kappa}}(c)\ac{1}{\eta} = 0
\]
at order $t$, and is therefore satisfied in the sense we require at this point. Note that it is satisfied exactly at $c=c^*$, since $\ell_{\bm{\kappa}}(c^*)=0$.

\subsubsection*{Terms of order $t^2$:} For the second-order terms, we apply the same overall strategy. The main difference is that the resulting equations are significantly more involved.

This time, the conclusion from~\eqref{eq:flattened_euler} is that
\[
    \ac{2}{\bm{u}} = -B[\ac{1}{\eta}]\ac{1}{\bm{u}} - \parn[\big]{B[\ac{2}{\eta}] + M[\ac{1}{\eta}]}\ac{0}{\bm{u}} + \frac{1}{c}h'(\ac{0}{q})\ac{2}{q}\bm{e}_1 + \nabla \ac{2}{\phi} + \ac{2}{U}\bm{e}_1
\]
where
\[
    M[\dot{\eta}] \ceq \parn[\Big]{\nabla_{\bm{x}'} \pi[\dot{\eta}] \otimes \nabla_{\bm{x}'} \pi[\dot{\eta}] + \frac{1}{d^2}\dot{\eta}^2\id_{\R^2}} \oplus 0,
\]
and $\ac{2}{\phi}$ is $\Lambda$-periodic and $\ac{2}{U}$ is constant, like before. Now, one may check that the identity
\[
    B[\ac{1}{\eta}]\ac{1}{\bm{u}} + M[\ac{1}{\eta}]\ac{0}{\bm{u}} = B[\ac{1}{\eta}]\nabla \ac{1}{\phi}
\]
holds, which simplifies the expression for $\ac{2}{\bm{u}}$.

At the second order,~\eqref{eq:flattened_time} and \eqref{eq:flattened_time_boundary} read
\[
    c \partial_1 \ac{2}{\tau} = - \ac{1}{\bm{u}} \cdot \nabla \ac{1}{\tau} + \frac{1}{d}\ac{2}{\eta} - \frac{1}{c}\ac{2}{u}_1, \qquad \text{and} \qquad \rest[\big]{\ac{2}{\tau}}_{x_1=0} = 0,
\]
where we compute that
\begin{align*}
    \ac{2}{u}_1                              & = - (B[\ac{1}{\eta}]\nabla \ac{1}{\phi})_1 + \frac{c}{d}\ac{2}{\eta} + \frac{1}{c}h'(\ac{0}{q})\ac{2}{q} + \partial_1 \ac{2}{\phi} + \ac{2}{U} \\
    \ac{1}{\bm{u}} \cdot \nabla \ac{1}{\tau} & = \frac{1}{c}(B[\ac{1}{\eta}]\nabla \ac{1}{\phi})_1 - \frac{1}{c^2}\abs{\nabla \ac{1}{\phi}}^2,
\end{align*}
in particular leading to
\begin{equation}
    \label{eq:q2}
    \begin{aligned}
        Q'(c)\ac{2}{q} & = \int_0^{\lambda_1} \abs*{\nabla(\partial_1 \ac{1}{\eta}\mf{f}_{\abs{\bm{\kappa}}})}^2\,dx_1 - \ac{0}{q}\ac{2}{U}                                                                                                                                                   \\
                       & =\frac{\pi \kappa_1}{2}\parn[\Big]{\abs{\bm{\kappa}}^2\mf{f}_{\abs{\bm{\kappa}}}^2 + (\mf{f}'_{\abs{\bm{\kappa}}})^2 + \brak[\big]{(\kappa_1^2 - \kappa_2^2)\mf{f}_{\abs{\bm{\kappa}}}^2 + (\mf{f}'_{\abs{\bm{\kappa}}})^2}\cos(2\kappa_2 x_2)} - \ac{0}{q}\ac{2}{U}
    \end{aligned}
\end{equation}
by integration. This can be solved for $\ac{2}{q}$, since we have assumed that $Q'(c) \neq 0$ for all $c \in W$.

We can now use~\eqref{eq:integral_cond_flattened} at the second order to determine $\ac{2}{U}$, and therefore $\ac{2}{q}$, because the terms with $\ac{2}{q}$ cancel. It demands that
\[
    0 = \fint_{\Omega_0^0} \parn[\Big]{\ac{2}{u}_1 - \frac{c}{d}\ac{2}{\eta}}\,d\bm{x} = \fint_{\Omega_0^0} \parn[\bigg]{\frac{1}{c}h'(\ac{0}{q})\ac{2}{q}-(B[\ac{1}{\eta}]\ac{1}{\phi})_1} + \ac{2}{U}=\frac{c}{Q'(c)}\parn[\bigg]{\ac{2}{U}-\frac{c\kappa_1^2}{4d}\mo{\mf{f}}_{\abs{\bm{\kappa}}}},
\]
or that
\[
    \ac{2}{U} = \frac{c\kappa_1^2}{4d}\mo{\mf{f}}_{\abs{\bm{\kappa}}}.
\]

For~\eqref{eq:flattened_divergence} and \eqref{eq:flattened_kinematic} to be satisfied, we must have that
\begin{alignat*}{3}
    \Delta \ac{2}{\phi}     & = \nabla \cdot \parn[\big]{B[\ac{1}{\eta}]\nabla \ac{1}{\phi}}           & \qquad & \text{in } \Omega_0^0,          \\
    \partial_3 \ac{2}{\phi} & = (B[\ac{1}{\eta}]\nabla \ac{1}{\phi})_3 + c\partial_1 \pi[\ac{2}{\eta}] &        & \text{on } \partial \Omega_0^0,
\end{alignat*}
which may be written as
\begin{alignat*}{3}
    \Delta \parn[\big]{\ac{2}{\phi}- \partial_3 \ac{1}{\phi} \pi[\ac{1}{\eta}]}     & =0                                                                                       & \qquad & \text{in }\Omega_0^0,          \\
    \partial_3 \parn[\big]{\ac{2}{\phi}- \partial_3 \ac{1}{\phi} \pi[\ac{1}{\eta}]} & =c\partial_1\pi\brak[\big]{\ac{2}{\eta} + \mo{\mf{f}}_{\abs{\bm{\kappa}}}\ac{1}{\alpha}} &        & \text{on } \partial\Omega_0^0,
\end{alignat*}
where
\[
    \ac{1}{\alpha} \ceq \frac{1}{2}\parn[\big]{(\partial_2 \ac{1}{\eta})^2-(\kappa_1^2 + \abs{\bm{\kappa}}^2)\ac{1}{\eta}^2},
\]
so that
\[
    \ac{2}{\phi} = \partial_3 \ac{1}{\phi} \pi[\ac{1}{\eta}] + c \partial_1 \chi\brak[\big]{\ac{2}{\eta} + \mo{\mf{f}}_{\abs{\bm{\kappa}}}\ac{1}{\alpha}}.
\]

Finally, the second order dynamic boundary condition~\eqref{eq:dynamic_bdry_flattened} becomes
\begin{equation}
    \label{eq:dynamic_bdry_flattened_second_order}
    \frac{1}{2}\abs{\ac{1}{\bm{u}}}^2 + c\ac{2}{u}_1 + \frac{c^2}{2}(\partial_1 \ac{1}{\eta})^2 - 2c \frac{\ac{1}{\eta}}{d}\ac{1}{u}_1 + \parn[\Big]{\frac{3\ac{1}{\eta}^2}{2d^2}-\frac{\ac{2}{\eta}}{d}}c^2 + g\ac{2}{\eta}-\sigma\Delta \ac{2}{\eta}-h'(\ac{0}{q})\ac{2}{q}=0
\end{equation}
at $x_3 = 0$, after using that $\ac{0}{u} =c\bm{e}_1$. Now
\begin{align*}
    c\ac{2}{u}_1 - \frac{c^2}{d}\ac{2}{\eta} - h'(\ac{0}{q})\ac{2}{q} & =-c(B[\ac{1}{\eta}]\nabla \ac{1}{\phi})_1 + c\partial_1 \ac{2}{\phi} + c\ac{2}{U}                                                                                                       \\
                                                                      & =\frac{c}{d}\ac{1}{\eta}\partial_1 \ac{1}{\phi}-c^2\kappa_1^2\ac{1}{\eta}^2+ c\ac{2}{U}+c^2 \partial_1^2 \chi\brak[\big]{\ac{2}{\eta} + \mo{\mf{f}}_{\abs{\bm{\kappa}}}\ac{1}{\alpha}},
\end{align*}
at $x_3 = 0$, while
\[
    \frac{1}{2}\abs{\ac{1}{\bm{u}}}^2 + \frac{c^2}{2}(\partial_1 \ac{1}{\eta})^2 - 2c\frac{\ac{1}{\eta}}{d}\ac{1}{u}_1 + \frac{3c^2}{2d^2}\ac{1}{\eta}^2 = - \frac{c}{d}\ac{1}{\eta}\partial_1 \ac{1}{\phi} + \frac{1}{2}\abs{\nabla \ac{1}{\phi}}^2.
\]
Thus~\eqref{eq:dynamic_bdry_flattened_second_order} simplifies to
\[
    \frac{1}{2}\abs{\nabla\ac{1}{\phi}}^2 -c^2\kappa_1^2 \ac{1}{\eta}^2 + c\ac{2}{U}+c^2 \partial_1^2 \chi\brak[\big]{\ac{2}{\eta} + \mo{\mf{f}}_{\abs{\bm{\kappa}}}\ac{1}{\alpha}} +g\ac{2}{\eta}-\sigma \Delta \ac{2}{\eta} =0,
\]
where
\begin{align*}
    \ac{1}{\alpha}                         & =-\frac{1}{4}(1+\cos(2\kappa_1 x_1))(\kappa_1^2 + \abs{\bm{\kappa}}^2\cos(2\kappa_2 x_2)),                                                                     \\
    \frac{1}{2}\abs{\nabla \ac{1}{\phi}}^2 & =
    \begin{multlined}[t]
        \frac{c^2\kappa_1^2\mo{\mf{f}}_{\abs{\bm{\kappa}}}^2}{8}\parn[\Big]{\abs{\bm{\kappa}}^2\parn[\big]{1+\cos(2\kappa_1x_1)\cos(2\kappa_2 x_2)} + (\kappa_1^2-\kappa_2^2)\parn[\big]{\cos(2\kappa_1 x_1)+\cos(2\kappa_2 x_1)}}\\
        +\frac{c^2\kappa_1^2}{8}\parn[\big]{1-\cos(2\kappa_1 x_1)+\cos(2\kappa_2 x_2)-\cos(2\kappa_1 x_1)\cos(2\kappa_2 x_2)},
    \end{multlined} \\
    c^2\kappa_1^2 \ac{1}{\eta}^2           & = \frac{c^2\kappa_1^2}{4}\parn[\big]{1+\cos(2\kappa_1 x_1) + \cos(2\kappa_2 x_2)+\cos(2\kappa_1 x_1)\cos(2\kappa_2 x_2)},
\end{align*}
so that
\[
    \ac{2}{\eta} = \frac{c^2\kappa_1^2}{8}\parn[\Big]{\frac{a_{\bm{0}}}{\ell_{\bm{0}}(c) }+\frac{a_{2\kappa_1\bm{e}_1}}{\ell_{2\kappa_1\bm{e}_1}(c) }\cos(2\kappa_1 x_1)+\frac{a_{2\kappa_2\bm{e}_2}}{\ell_{2\kappa_2\bm{e}_2}(c) }\cos(2\kappa_2 x_2)+\frac{a_{2\bm{\kappa}}}{\ell_{2\bm{\kappa}}(c) }\cos(2\kappa_1 x_1)\cos(2\kappa_2 x_2)}
\]
with
\begin{align*}
    a_{\bm{0}}            & = 1 - \mo{\mf{f}}_{\abs{\bm{\kappa}}}^2\abs{\bm{\kappa}}^2 - \frac{2}{d}\mo{\mf{f}}_{\abs{\bm{\kappa}}}                                                                                                   \\
    a_{2\kappa_1\bm{e}_1} & =3 - \mo{\mf{f}}_{\abs{\bm{\kappa}}}^2(\kappa_1^2-\kappa_2^2) - 8\kappa_1^2 \mo{\mf{f}}_{\abs{\bm{\kappa}}}\mo{\mf{f}}_{2\kappa_1}                                                                        \\
    a_{2\kappa_2\bm{e}_2} & =1-\mo{\mf{f}}_{\abs{\bm{\kappa}}}^2(\kappa_1^2-\kappa_2^2)                                                                                                                                               \\
    a_{2\bm{\kappa}}      & =3-\mo{\mf{f}}_{\abs{\bm{\kappa}}}^2\abs{\bm{\kappa}}^2 - 8\abs{\bm{\kappa}}^2 \mo{\mf{f}}_{\abs{\bm{\kappa}}} \mo{\mf{f}}_{2\abs{\bm{\kappa}}}=1-3\mo{\mf{f}}_{\abs{\bm{\kappa}}}^2 \abs{\bm{\kappa}}^2.
\end{align*}
\subsubsection*{Expansion of c}
By taking the projection of the dynamic boundary condition onto the space spanned by $\ac{1}{\eta}$, we obtain the equation
\[
    \ell_{\bm{\kappa}}(c)\ac{1}{\eta} t+o(t^2)=0,
\]
meaning that
\begin{equation}
    \label{eq:c_expansion}
    c=c^*+o(t),
\end{equation}
allowing us to prove the following result.

\begin{proposition}
    \label{prop:nonzero_vorticity}
    For the solutions to~\eqref{eq:kinematic}, \eqref{eq:unflattened_problem} obtained in \Cref{thm:Main_result} we have
    \[
        \bm{\omega}[c(t),t] = \frac{h'(\ac{0}{q}[c^*])}{c^*} \parn[\big]{\partial_3 \ac{2}{q}[c^*]\bm{e}_2 - \partial_2 \ac{2}{q}[c^*]\bm{e}_3}t^2 + o(t^2)
    \]
    as $t \to 0$, where $\ac{2}{q}$ is given in~\eqref{eq:q2}. In particular, they are rotational for all $0 < \abs{t} \ll 1$ if $h'(\ac{0}{q}[c^*]) \neq 0$.
\end{proposition}
\begin{proof}
    The right-hand side of~\eqref{eq:flattened_euler} is equal to $\bm{\omega}$ by~\eqref{eq:curl_flattening}. We begin by considering the factor $\nabla q[c,t]$ in isolation. The expansion in~\eqref{eq:sol_expansion_1} gives
    \begin{align*}
        \nabla q[c,t]    & =t^2\nabla \ac{2}{q}[c]+o(t^2),
        \shortintertext{and thus}
        \nabla q[c(t),t] & =t^2\nabla \ac{2}{q}[c^*]+o(t^2)
    \end{align*}
    by~\eqref{eq:c_expansion}. It immediately follows that
    \begin{align*}
        \bm{\omega}[c(t),t] & = h'(q[c,t])\nabla q[c,t]\times \nabla \tau[c,t]                                  \\
                            & =t^2 h'(\ac{0}{q}[c^*])\nabla \ac{2}{q}[c^*]\times \nabla\ac{0}{\tau}[c^*]+o(t^2)
    \end{align*}
    as $t \to 0$ by using the expansions in~\eqref{eq:sol_expansion_1}. The result now follows upon using the expressions for $\ac{0}{\tau}$, $\ac{0}{q}$ from~\eqref{eq:trivial_time_functions}, and $\ac{2}{q}$ from~\eqref{eq:q2}.
\end{proof}

\subsection{Bernoulli surfaces}
The level surfaces of the Bernoulli function $H$, known as Bernoulli surfaces, are interesting in view of Arnold's structure theorem~\cite[Ch. II, §1]{Arnold98Topological}. This theorem roughly says that the Bernoulli surfaces are either invariant tori or diffeomorphic to the annulus $\R \times S^1$. The latter happens when a surface touches the boundary. In our case, everything has to be understood by taking the quotient $\mb R^2/\Lambda$. In general, the streamlines on the annuli are closed (or rather periodic in our setting), while those on the tori are either all closed or dense in the surface. However, the latter cannot happen in our case due to the symmetry assumptions, which guarantee periodicity of the streamlines.

The main conclusion of this section is that for the fluid flows constructed in this paper, both tori and annuli appear as Bernoulli surfaces. This is different from two-dimensional water waves with vorticity, where the surface and bottom are always Bernoulli surfaces, and all other Bernoulli surfaces are tori strictly contained in the interior of the fluid domain.

In this section, we always work under the assumption $h'(\ac{0}{q}[c^*]) \neq 0$. Due to the relationship between $q$ and the Bernoulli function $H$, the Bernoulli surfaces are then also level surfaces of $q$. By combining~\eqref{eq:sol_expansion_1} and \eqref{eq:c_expansion}, we see that $q$ can be expanded with respect to $t$ as
\begin{equation}
    \label{eq:q_expansion}
    q[c(t),t](\bm{x})=\ac{0}{q}[c(t)]+\ac{2}{q}[c^*](\bm{x}) t^2 + o[t^2](\bm{x})
\end{equation}
in $\mb{Y}_+^{\mf{s}-3}(\ol{\Omega^0})$ as $t \to 0$, where we recall that $\mf{s}>4$. In view of~\eqref{eq:q_expansion}, we expect that the level surfaces of $q[c(t),t]$ are closely related to those of $\ac{2}{q}[c^*]$, and should therefore investigate these. To facilitate this, we will treat $\ac{2}{q}[c^*]$ as extended to the larger domain $\ol{\Omega^\infty}$, and define
\[
    K_{x_3}^\pm = \abs{\bm{\kappa}}^2\mf{f}_{\abs{\bm{\kappa}}}^2 + (\mf{f}'_{\abs{\bm{\kappa}}})^2 \pm \brak[\big]{(\kappa_1^2 - \kappa_2^2)\mf{f}_{\abs{\bm{\kappa}}}^2 + (\mf{f}'_{\abs{\bm{\kappa}}})^2}
\]
for $x_3 \in [-d,0]$.
\begin{lemma}
    \label{lem:q2_level_curves}
    The function $\ac{2}{q}[c^*]$ is increasing with respect to $x_3$, and its image on $\ol{\Omega^0}$ is the interval $\mf{q}(\ol{I})$, where $I = (\ul{K},\ol{K})$,
    \[
        \mf{q}(K) \ceq \frac{\pi \kappa_1}{2Q'(c^*)}\parn[\bigg]{K - \frac{\mo{\mf{f}}_{\abs{\bm{\kappa}}}}{d}}, \quad \ul{K} \ceq \min(K^+_{-d},K^-_{-d}), \quad \text{and} \quad \ol{K} \ceq \max(K^+_0,K^-_0).
    \]
    For each $K \in I$, the level surface $\ac{2}{q}[c^*]^{-1}(\mf{q}(K))$ is the graph of a smooth function $\map{\psi_0^K}{\mc{D}_0^K}{[-d,\infty)}$, only depending on $x_2$, where the domain
    \[
        \mc{D}_0^K \ceq \brac[\big]{\bm{x}' : (\bm{x}',x_3) \in \ac{2}{q}[c^*]^{-1}(\mf{q}(K))}
    \]
    is nondecreasing in $K$. Also, if $K_1 < K_2$, then $\psi_0^{K_1} < \psi_0^{K_2}$ on $\mc{D}_0^{K_1}$.

    There always exist annuli, in the sense that $\ac{2}{q}[c^*]^{-1}(\mf{q}(K)) \not\subset \Omega^0$ for some $K \in I$. Moreover, tori exist precisely when
    \begin{equation}
        \label{eq:tori_condition_kappa}
        \frac{\abs{\bm{\kappa}}^2}{1+\cosh(\abs{\bm{\kappa}}d)^2} < \kappa_2^2 < \frac{\abs{\bm{\kappa}}^2\cosh(2\abs{\bm{\kappa}}d)}{1+\cosh(\abs{\bm{\kappa}}d)^2},
    \end{equation}
    where the final inequality is trivially satisfied when $\abs{\bm{\kappa}}d \geq \log(1+\sqrt{2})$.
\end{lemma}
\begin{proof}
    The proof of the first part is essentially immediate from~\eqref{eq:q2}, after noting that
    \begin{equation}
        \label{eq:q2_x3_derivative}
        \frac{Q'(c^*)}{\pi \kappa_1} \partial_3 \ac{2}{q}[c^*] = (\mf{f}_{\abs{\bm{\kappa}}}^2)'(\abs{\bm{\kappa}}^2 + \kappa_1^2 \cos(2\kappa_2x_2))
    \end{equation}
    is strictly positive on $\Omega^\infty$.

    For the final part, we see that $K$ corresponds to a torus exactly when
    \[
        K \in \mo{I} \ceq \parn[\big]{\max(K_{-d}^+,K_{-d}^-),\min(K_0^+,K_0^-)},
    \]
    whence tori exist if and only $\mo{I} \neq \varnothing$, or equivalently
    \begin{equation}
        \label{eq:tori_condition}
        \min(\kappa_2^2 \cosh(\abs{\bm{\kappa}d})^2,\kappa_1^2 \cosh(\abs{\bm{\kappa}}d)^2 + \abs{\bm{\kappa}}^2 \sinh(\abs{\bm{\kappa}}d)^2) > \max(\kappa_1^2,\kappa_2^2),
    \end{equation}
    is satisfied. In order to demonstrate that this in turn is equivalent to~\eqref{eq:tori_condition_kappa}, we treat three regimes separately:

    If $\kappa_2^2 < \kappa_1^2$, then~\eqref{eq:tori_condition} reads $\kappa_1^2 < \kappa_2^2 \cosh(\abs{\bm{\kappa}} d)^2$. Next, if $\kappa_1^2 \leq \kappa_2^2 \leq \kappa_1^2 \cosh(2\abs{\bm{\kappa}}d)$, then it is trivially satisfied. Finally, if $\kappa_1^2 \cosh(2\abs{\bm{\kappa}}d) < \kappa_2^2$ we find that~\eqref{eq:tori_condition} becomes $(2\kappa_1^2 + \kappa_2^2)\cosh(2\abs{\bm{\kappa}}d) > 3\kappa_2^2$. In summary, the condition is satisfied if and only if
    \[
        \frac{\kappa_1^2}{\cosh(\abs{\bm{\kappa}}d)^2} < \kappa_2^2 < \frac{2\kappa_1^2 + \kappa_2^2}{3}\cosh(2\abs{\bm{\kappa}}d),
    \]
    which turns into~\eqref{eq:tori_condition_kappa} after eliminating $\kappa_1$.

    Lastly, we claim that $\mo{I} \subsetneq I$, and therefore that there are always annuli. Indeed, $K_{x_3}^+ = K_{x_3}^-$ if and only if $x_3 $ is a root of
    \[
        (\kappa_1^2 - \kappa_2^2)\mf{f}_{\abs{\bm{\kappa}}}^2 + (\mf{f}'_{\abs{\bm{\kappa}}})^2 = \frac{\kappa_1^2 \cosh(2\abs{\bm{\kappa}}(x_3+d))-\kappa_2^2}{\abs{\bm{\kappa}}^2\sinh(\abs{\bm{\kappa}}d)^2},
    \]
    and there is clearly at most one such root on $[-d,0]$.
\end{proof}

\begin{figure}
    \centering
    \begin{subfigure}[t]{0.5\textwidth}
        \tikzsetnextfilename{contours_annuli}
        \begin{tikzpicture}
            \begin{axis}[xmin=-pi/(ln(1+sqrt(2))/2),xmax=pi/(ln(1+sqrt(2))/2),xtick={-pi/(ln(1+sqrt(2))/2),0,pi/(ln(1+sqrt(2))/2)},xticklabels={$-\pi/\kappa_2$,$0$,$\pi/\kappa_2$},ymin=-1,ymax=0,ytick=\empty,width=\textwidth,restrict y to domain=-1:0]
                \addplot[contour prepared, contour prepared format=matlab, contour/labels=false, contour/draw color=black] table {contours_annuli.dat};
            \end{axis}
        \end{tikzpicture}
        \subcaption{When $\kappa_2 = \abs{\bm{\kappa}}/2$ there are only annuli.}
    \end{subfigure}%
    \begin{subfigure}[t]{0.5\textwidth}
        \tikzsetnextfilename{contours_tori}
        \begin{tikzpicture}
            \begin{axis}[xmin=-pi/(ln(1+sqrt(2))/sqrt(2)),xmax=pi/(ln(1+sqrt(2))/sqrt(2)),xtick={-pi/(ln(1+sqrt(2))/sqrt(2)),0,pi/(ln(1+sqrt(2))/sqrt(2))},xticklabels={$-\pi/\kappa_2$,$0$,$\pi/\kappa_2$},ymin=-1,ymax=0,ytick=\empty,width=\textwidth,restrict y to domain=-1:0]
                \addplot[contour prepared, contour prepared format=matlab, contour/labels=false, contour/draw color=black] table {contours_tori.dat};
            \end{axis}
        \end{tikzpicture}
        \subcaption{When $\kappa_2 = \abs{\bm{\kappa}}/\sqrt{2}$ there are both annuli and tori.}
    \end{subfigure}
    \caption{Level curves of $\protect\ac{2}{q}[c^*]$ when $\abs{\bm{\kappa}} = \log(1+\sqrt{2})$ and $d=1$. In this situation~\eqref{eq:tori_condition_kappa} reads $\abs{\bm{\kappa}}^2 / 3 < \kappa_2^2 < \abs{\bm{\kappa}}^2$}
    \label{fig:bernoulli_surfaces}
\end{figure}

\Cref{fig:bernoulli_surfaces} illustrates \Cref{lem:q2_level_curves} for specific wave numbers. We will now use~\eqref{eq:q_expansion} and \Cref{lem:q2_level_curves} to describe the Bernoulli surface of $q$ for sufficiently small $t$. There is a slight complication related to the flattening of the domain, which we will briefly discuss in \Cref{rem:flattened_bernoulli_surface}.

\begin{proposition}
    \label{prop:bernoullisurf}
    For every $\epsilon > 0$ and $K \in I$, the Bernoulli surface
    \begin{equation}
        \label{eq:bernoulli_surface_defining}
        q[c[t],t] = \ac{0}{q}[c(t)] + \mf{q}(K) t^2
    \end{equation}
    on $x_3 \geq -d + \epsilon$ is the graph of a function $\map{\psi_t^K}{\mc{D}_t^K}{[-d+\epsilon,0]}$ for all sufficiently small $t \neq 0$. Moreover
    \[
        \ol{\liminf_{t \to 0}{\mc{D}_t^K}}\ceq (\psi_0^K)^{-1}([-d+\epsilon,0]),
    \]
    and
    \[
        \lim_{t \to 0} \psi_t^K = \psi_0^K
    \]
    locally uniformly on the interior of $ \liminf_{t \to 0}{\mc{D}_t^K}$. If $K \notin \ol{I}$, then the Bernoulli surface is empty for sufficiently small $t$.

    The final part of \Cref{lem:q2_level_curves} carries over: If~\eqref{eq:tori_condition_kappa} is satisfied, then there are tori for sufficiently small $t$. If it is not satisfied, with the possible exception of equality, then there are no tori for sufficiently small $t$. There are always annuli.
\end{proposition}
\begin{remark}
    \label{rem:flattened_bernoulli_surface}
    In fact, each $ \psi_t^K $ describes a Bernoulli surface in the flattened variables. However, due to the way scalar functions like $q$ transform, we find that the corresponding Bernoulli surface in physical coordinates is given by
    \[
        \hat{\psi}_t^K = \psi_t^K + \eta[t]\parn[\bigg]{1 + \frac{\psi_t^K}{d}}
    \]
    instead. Because $\eta[t] = o[t]$, this means that $ \psi^K_0 $ also approximates the Bernoulli surfaces in physical coordinates.
\end{remark}
\begin{proof}
    From~\eqref{eq:q2_x3_derivative} we not only get that $\partial_3 \ac{2}{q}[c^*]$ is positive on $\Omega^\infty$, but that it is bounded away from zero when $x_3 \geq -d + \epsilon$. Consequently the same is true for $\partial_3 q[c(t),t]$ for sufficiently small $t \neq 0$ by~\eqref{eq:q_expansion}. The implicit function theorem now establishes the existence of $\psi_t^K$, with regularity in both space and time. Furthermore,~\eqref{eq:bernoulli_surface_defining} may be written
    \[
        \ac{2}{q}[c^*](\bm{x}) + o[1](\bm{x}) = \mf{q}(K),
    \]
    which yields the convergence as $t \to 0$.

    The final part holds because we can always choose $\epsilon$ small enough for any given torus for $\ac{2}{q}[c^*]$ to be entirely contained in $x_3 \geq -d + \epsilon$. The exception for equality stems from the case of a degenerate annulus that can (potentially) become a torus when $t \neq 0$. The existence of annuli follows from the fact that $\ac{2}{q}[c^*]$ cannot be constant both on the bottom and the surface, and hence the same holds for $q[c(t),t]$ for sufficiently small $t$.
\end{proof}
\appendix
\section{Formal differentiability of compositions}
\label{Appendix:composition}
Let $(X_s)$, $(Y_s)$, and $(Z_s)$ be decreasing scales of Banach spaces, and fix an index $\mf{s} \in \R$. Suppose that $\mc{X}_{\mf{s}} \subset X_{\mf{s}}$ and $\mc{Y}_{\mf{s}} \subset Y_{\mf{s}}$ are subsets, on which we have maps
\[
    \mc{X}_{\mf{s}} \xrightarrow{y} \mc{Y}_{\mf{s}} \xrightarrow{z} Z_{\mf{s}},
\]
with formal derivatives
\begin{align*}
    D^j y\colon \mc{X}_{\mf{s}} & \to \mc{L}^{ j }\parn[\big]{X_{\mf{s}-p_4^j},Y_{\mf{s}-p_1^j}}, \\
    D^j z\colon \mc{Y}_{\mf{s}} & \to \mc{L}^{ j }\parn[\big]{Y_{\mf{s}-q_4^j},Z_{\mf{s}-q_1^j}},
\end{align*}
for $j=1,2$. Furthermore, assume that these operators satisfy
\begin{align}
    \norm{D^j y[x]}_{\mc{L}^{ j }\parn[\big]{X_{s-p_4^j},Y_{s-p_1^j}}}                          & \lesssim 1,                             &                       & s_1^j< s \leq \mf{s},
    \label{eq:app_y_bdd}                                                                                                                                                                  \\
    \norm{D^j z[y]}_{\mc{L}^{ j }\parn[\big]{Y_{s-q_4^j},Z_{s-q_1^j}}}                          & \lesssim 1,                             &                       & t_1^j< s \leq \mf{s},
    \label{eq:app_z_bdd}                                                                                                                                                                  \\
    \norm{\Delta_{h} D^j y[x]}_{\mc{L}^j\parn[\big]{X_{s-p_6^j},Y_{s-p_2^j}}}                   & \lesssim\norm{h}_{X_{s-p_5^j}},
                                                                                                &                                         & s_2^j< s \leq \mf{s},
    \label{eq:app_y_lip}                                                                                                                                                                  \\
    \norm{\Delta_{k} D^j z[y]}_{\mc{L}^{ j}\parn[\big]{Y_{s-q_6^j},Z_{s-q_2^j}}}                & \lesssim \norm{ k}_{Y_{s-q_5^j}},
                                                                                                &                                         & t_2^j< s \leq \mf{s},
    \label{eq:app_z_lip}                                                                                                                                                                  \\
    \shortintertext{for $j=0,1$, and}
    \norm{ \Delta_{h} D^j y[x]-D^{j+1} y[x]h}_{\mc{L}^{j}\parn[\big]{X_{s-p_8^j},Y_{s-p_3^j}}}  & =o\parn[\big]{\norm{ h}_{X_{s-p_7^j}}},
                                                                                                &                                         & s_3^j< s \leq \mf{s},
    \label{eq:app_y_diff}                                                                                                                                                                 \\
    \norm{ \Delta_{k} D^j z[y]-D^{j+1} z[y]k}_{\mc{L}^{j }\parn[\big]{Y_{s-q_8^j},Z_{s-q_3^j}}} & =o\parn[\big]{ \norm{k}_{Y_{s-q_7^j}}},
                                                                                                &                                         & t_3^j< s \leq \mf{s},
    \label{eq:app_z_diff}
\end{align}
for $j=0$.
Here $p_i^j, q_i^j \ge 0$ and these properties are assumed to hold uniformly for all $x, x+ h\in \mc{X}_{\mf{s}}$ and $y,y+k\in \mc{Y}_{\mf{s}}$.

We denote the composition $\map{z \circ y}{\mc{X}_\mf{s}}{Z_\mf{s}}$ by $f$, and introduce its formal derivatives
\begin{equation}
    \label{eq:faa_di_bruno}
    D^jf[x](h_i)_{i \in I_j} \ceq \sum_{\pi \in \mc{P}_j} D^{\abs{\pi}} z[y[x]]\parn[\big]{D^{\abs{\beta}}y[x](h_i)_{i \in \beta}}_{\beta \in \pi}
\end{equation}
through Faà di Bruno's formula, whenever this definition makes sense. Here $I_j \ceq \brac{1,\ldots,j}$, and $\mc{P}_j$ is the set of partitions of $I_j$.
\begin{proposition}
    \label{Prop:composition}
    If $p_1^0=q_1^0=0$, then for all $x,x+h\in \mc{X}_{\mf{s}}$, the operators $D^j f$ satisfy:
    \begin{equation}
        \label{eq:app_f_bdd}
        \norm{ D^jf[x]}_{\mc{L}^j\parn[\big]{X_{s-r_4^j}, Z_{s-r_1^j}}}\lesssim 1,\qquad \varsigma_1^j< s\leq \mf{s},
    \end{equation}
    for $j = 0$ if
    \begin{align*}
        r_1^0         & =0,                                                \\
        \varsigma_1^0 & = t_1^0,                                           \\
        \shortintertext{and $j=1$ if $\delta^1 \ceq p_1^1-q_4^1$ and}
        r_1^1         & =q_1^1 + \delta_+^1,                               \\
        r_4^1         & = p_4^1 + \delta_-^1,                              \\
        \varsigma_1^1 & =\max\brac{t_1^1 + \delta_+^1,s_1^1 + \delta_-^1}.
    \end{align*}
    \begin{equation}
        \label{eq:app_f_lip}
        \norm{ \Delta_h D^j f[x]}_{\mc{L}^j\parn[\big]{X_{s-r_6^j},Z_{s-r_2^j}}}\lesssim \norm{ h}_{X_{s-r_5^j}},\qquad \varsigma_2^j<s\leq \mf{s},
    \end{equation}
    for $j=0$ if $\delta^2 \ceq p_2^0 - q_5^0$ and
    \begin{align*}
        r_2^0         & =q_2^0 + \delta^2_+                                                                                                                                         \\
        r_5^0         & = p_5^0 + \delta^2_-,                                                                                                                                       \\
        \varsigma_2^0 & =\max\brac{t_2^0 + \delta^2_+,s_2^0 + \delta_-^2}                                                                                                           \\
        \shortintertext{and $j=1$ if}
        r_2^1         & =\max\brac{q_2^1,\overbrace{q_2^1+p_2^0-q_5^1}^{\varkappa^1},\overbrace{q_2^1+p_1^1-q_6^1}^{\varkappa^2},q_1^1,\overbrace{q_1^1+p_2^1-q_4^1}^{\varkappa^3}} \\
        r_5^1         & = r_2^1 + \min\brac{p_5^0 -\varkappa^1,p_5^1-\varkappa^3},                                                                                                  \\
        r_6^1         & = r_2^1 + \min\brac{p_4^1 -\varkappa^2,p_6^1-\varkappa^3},                                                                                                  \\
        \varsigma_2^1 & =r_2^1+ \max\brac{t_2^1-q_2^1,s_2^0-\varkappa^1,s_1^1-\varkappa^2,t_1^1-q_1^1,s_2^1-\varkappa^3}.
    \end{align*}
    \begin{equation}
        \label{eq:app_f_diff}
        \norm{\Delta_{h}D^j f[x]-D^{j+1} f[x]h}_{\mc{L}\parn[\big]{X_{s-r_8^j},Z_{s-r_3^j}}}=o\parn[\big]{\norm{ h}_{X_{s-r_7^j}}},\qquad \varsigma_3^j< s\leq \mf{s},
    \end{equation}
    for $j=0$ if
    \begin{align*}
        r_3^0         & =\max\brac{q_3^0,\overbrace{q_3^0+p_2^0-q_7^0}^{\varkappa^4},q_1^1,\overbrace{q_1^1+p_3^0-q_4^1}^{\varkappa^5}}, \\
        r_7^0         & =r_3^0+\min\brac{p_5^0-\varkappa^4,p_7^0-\varkappa^5},                                                           \\
        \varsigma_3^0 & =r_3^0+\max\brac{t_3^0-q_3^0,s_2^0-\varkappa^4,t_1^1- q_1^1,s_3^0- \varkappa^5}.
    \end{align*}
\end{proposition}
\begin{proof}
    The first claim for $j=0$ follows immediately from~\eqref{eq:app_z_bdd} and the fact that $y$ maps $\mc{X}_{\mf{s}}$ into $\mc{Y}_{\mf{s}}$, while for $j=1$ we estimate
    \begin{align*}
        \norm{Dz[y[x]]Dy[x]h}_{Z_{s-r_1^1}} \qquad & \lesssimop_{\eqref{eq:app_z_bdd}}^{\mathclap{\substack{r_1^1\geq q_1^1          \\t_1^1+r_1^1-q_1^1 < s \leq \mf{s}}}}\qquad \norm{Dy[x]h}_{Y_{s-r_1^1+q_1^1-q_4^1}}\\
        \qquad                                     & \lesssimop_{\eqref{eq:app_y_bdd}}^{\mathclap{\substack{r_1^1\geq q_1^1+\delta^1 \\s_1^1+r_1^1-q_1^1-\delta^1 < s \leq \mf{s}}}}\qquad \norm{h}_{X_{s-r_1^1+q_1^1-p_4^1+\delta^1}}.
    \end{align*}

    For the second claim, we note that $\Delta_h f[x] = \Delta_{\Delta_h y[x]}z[y[x]]$, and therefore
    \begin{align*}
        \norm{\Delta_h f[x]}_{Z_{s-r_2^0}} \qquad & \lesssimop_{\eqref{eq:app_z_lip}}^{\mathclap{\substack{r_2^0\geq p_2^0          \\t_2^0+r_2^0-q_2^0 < s \leq \mf{s}}}}\qquad \norm{\Delta_h y[x]}_{Y_{s-r_2^0+q_2^0-q_5^0}}\\
        \qquad                                    & \lesssimop_{\eqref{eq:app_y_lip}}^{\mathclap{\substack{r_2^0\geq q_2^0+\delta^2 \\s_2^0+r_2^0- q_2^0-\delta^2 < s \leq \mf{s}}}}\qquad \norm{h}_{X_{s-r_2^0+q_2^0-p_5^0+\delta^2}},
    \end{align*}
    proving the claim for $j=0$. For $j=1$, we similarly note that
    \[
        \Delta_h Df[x] = Dz[y[x+h]]\Delta_h Dy[x] + \Delta_{\Delta_h y[x]}Dz[y[x]]Dy[x],
    \]
    where
    \begin{align*}
        \norm{Dz[y[x+h]]\Delta_h Dy[x]h_1}_{Z_{s-r_2^1}}\qquad & \lesssimop_{\eqref{eq:app_z_bdd}}^{\mathclap{\substack{r_2^1\geq q_2^1       \\t_1^1+r_2^1-q_1^1 < s \leq \mf{s}}}}\qquad \norm{\Delta_h Dy[x] h_1}_{Y_{s-r_2^1+q_1^1-q_4^1}}\\
        \qquad                                                 & \lesssimop_{\eqref{eq:app_y_lip}}^{\mathclap{\substack{r_2^1\geq \varkappa^3 \\s_2^1+r_2^1-\varkappa^3 < s \leq \mf{s}}}}\qquad \norm{h}_{X_{s-r_2^1-p_5^1+\varkappa^3}} \norm{h_1}_{X_{s-r_2^1-p_6^1+\varkappa^3}}
    \end{align*}
    and
    \begin{align*}
        \norm{\Delta_{\Delta_h y[x]}Dz[y[x]]Dy[x]h_1}_{Z_{s-r_2^1}}\qquad & \lesssimop_{\mathclap{\eqref{eq:app_z_lip}}}^{\mathclap{\substack{r_2^1\geq q_2^1                                                    \\t_2^1+r_2^1-q_2^1< s \leq \mf{s}}}}\qquad \norm{\Delta_hy[x]}_{Y_{s-r_2^1+q_2^1-q_5^1}}\norm{D y[x]h_1}_{Y_{s-r_2^1+q_2^1-q_6^1}}\\
        \qquad                                                            & \lesssimop_{\mathclap{\eqref{eq:app_y_bdd}, \eqref{eq:app_y_lip}}}^{\mathclap{\substack{r_2^1\geq \max\brac{\varkappa^1,\varkappa^2} \\r_2^1 + \max\brac{s_2^0 -\varkappa^1,s_1^1-\varkappa^2}< s \leq \mf{s}}}}\qquad \norm{h}_{X_{s-r_2^1-p_5^0 + \varkappa^1}}\norm{h_1}_{X_{s-r_2^1-p_4^1 + \varkappa^2}}.
    \end{align*}

    Finally, for the third claim, we write
    \[
        \Delta_h f[x] - Df[x]h =\Delta_{\Delta_hy[x]}z[y[x]]-Dz[y[x]\Delta_h y[x]] + Dz[y[x]]\parn*{\Delta_h y[x]-Dy[x]h},
    \]
    where
    \begin{align*}
        \norm{\Delta_{\Delta_hy[x]}z[y[x]]-Dz[y[x]\Delta_h y[x]]}_{Z_{s-r_3^0}} \qquad & \equalop_{\mathclap{\eqref{eq:app_z_diff}}}^{\mathclap{\substack{r_3^0\geq q_3^0      \\t_3^0+r_3^0- q_3^0< s \leq \mf{s}}}}\qquad o\parn[\big]{\norm{\Delta_h y[x]}_{Y_{s-r_3^0+q_3^0-q_7^0}}}\\
        \qquad                                                                         & \equalop_{\mathclap{\eqref{eq:app_y_lip}}}^{\mathclap{\substack{r_3^0\geq \varkappa^4 \\s_2^0 + r_3^0 - \varkappa^4 < s \leq \mf{s}}}}\qquad o\parn[\big]{\norm{h}_{X_{s-r_3^0-p_5^0+\varkappa^4}}}
    \end{align*}
    and
    \begin{align*}
        \norm{Dz[y[x]]\parn*{\Delta_h y[x]-Dy[x]h}}_{Z_{s-r_3^0}} \qquad & \lesssimop_{\mathclap{\eqref{eq:app_z_bdd}}}^{\mathclap{\substack{r_3^0\geq q_1^1      \\t_1^1 +r_3^0- q_1^1< s \leq \mf{s}}}}\qquad \norm{\Delta_hy-Dy[x]h}_{Y_{s-r_3^0+q_1^1-q_4^1}}\\
        \qquad                                                           & \equalop_{\mathclap{\eqref{eq:app_y_diff}}}^{\mathclap{\substack{r_3^0\geq \varkappa^5 \\s_3^0 + r_3^0 - \varkappa^5 < s \leq \mf{s}}}}\qquad o\parn[\big]{\norm{h}_{X_{s-r_3^0-p_7^0+\varkappa^5}}}.\qedhere
    \end{align*}
\end{proof}

We say that~\eqref{eq:app_y_bdd}--\eqref{eq:app_z_diff} hold for $k$ if they hold for all $j=0,\ldots,k-1$, and~\eqref{eq:app_y_bdd}--\eqref{eq:app_z_lip} also hold for $j=k$. Similarly, we say that~\eqref{eq:app_f_bdd}--\eqref{eq:app_f_diff} hold for $k$ if they hold for all $j=0,\ldots,k-1$, and~\eqref{eq:app_f_bdd} and \eqref{eq:app_f_lip} also hold for $j=k$. This allows us to extend \Cref{Prop:composition} to higher order derivatives in a special case.
\begin{corollary}
    \label{Cor:composition}
    If, in addition to the assumptions in \Cref{Prop:composition}, we have that~\eqref{eq:app_y_bdd}--\eqref{eq:app_z_diff} hold for $k$ with
    \[
        p_i^j = q_i^j =
        \begin{cases}
            i+j-1 & i \leq 3, \\
            1     & i \geq 4
        \end{cases}
        \qquad
        \text{and}
        \qquad
        s_i^j = t_i^j = i + j
    \]
    for all applicable $i$ and $j$, then~\eqref{eq:app_f_bdd}--\eqref{eq:app_f_diff} hold for $k$ with $r_i^j=p_i^j$ and $\varsigma_i^j=s_i^j$.
\end{corollary}
\begin{proof}
    For the case $k=1$, this follows immediately by applying \Cref{Prop:composition}. We proceed by showing that this result can be applied to the derivatives of $f$ up to order $k-1$, as defined by~\eqref{eq:faa_di_bruno}.

    From~\eqref{eq:faa_di_bruno} we know that, for $1 \leq m \leq k-1$, the derivative $D^m f[x](h_i)_{i \in I_m}$ is given by a sum of terms of the form
    \[
        D^{\abs{\pi}} z[y[x]]\parn[\big]{D^{\abs{\beta}}y[x](h_i)_{i \in \beta}}_{\beta \in \pi},
    \]
    where $\pi \in \mc{P}_m$. Any such term can be viewed as a composition of functions
    \[
        \mc{X}_\mf{s} \xrightarrow{\xi}\, \mc{O}_\mf{s} \xrightarrow{\zeta} \mc{L}^m(Z_{\mf{s}-1},Z_{\mf{s}-m}), \qquad
        \mc{O}_\mf{s} \subset Y_\mf{s} \times \prod_{\beta \in \pi} \mc{L}^{\abs{\beta}}(X_{\mf{s}-1},Y_{\mf{s}-\abs{\beta}})
    \]
    in a natural way.

    By the properties of the derivatives of $y$ and $z$, we find that $\xi$ and $\zeta$ satisfy~\eqref{eq:app_y_bdd}--\eqref{eq:app_z_diff} when $k=1$; with $p_i^j$, $q_i^j$ as before, but $s_i^j=t_i^j=j+i+m$. For $\xi$, and the first variable of $\zeta$, we immediately get the desired properties.
    Moreover, $\zeta$ is multilinear with respect to the other variables, and satisfies the inequality
    \[
        \norm{\zeta[y,(y_\beta)_{\beta \in \pi}]}_{\mc{L}^m(X_{s-1},Z_{s-m})} \lesssim \norm{D^{\abs{\pi}}z[y]}_{\mc{L}^{\abs{\pi}}(Y_{s-m+\abs{\pi}-1},Z_{s-m})} \prod_{\beta \in \pi} \norm{y_\beta}_{\mc{L}^{\abs{\beta}}(X_{s-1},Y_{s-\abs{\beta}})}
    \]
    where the norm in front of the product is $\lesssim 1$ for all $m+1<s\leq \mf{s}$ by~\eqref{eq:app_z_bdd}. Here we have used that
    \[
        m = \sum_{\beta \in \pi} \abs{\beta} \geq \max_{\beta \in \pi}{\abs{\beta}} + \abs{\pi}-1.
    \]

    Applying \Cref{Prop:composition} now tells us that~\eqref{eq:app_f_bdd}--\eqref{eq:app_f_diff} hold for $k=1$ for the composition of $\xi$ and $\zeta$ with $r_i^j=p_i^j$ and $\varsigma_i^j=j+i+m$, which concludes the proof.
\end{proof}
\begin{remark}
    \label{Rem:composition}
    Clearly, we can also use this result when one of the two mappings has better regularity properties; such as if $p_i^j \geq 1$ for $i\geq 4$, and all other parameters are unchanged. This can, for instance, be useful if the inner mapping is twice continuously differentiable from $\mc X_s$ to $\mc Y_s$ in the usual sense, with Lipschitz continuous derivatives. Then we could take all the parameters as above, but with $p_4^j=p_1^j$, $p_6^j=p_5^j=p_2^j$ and $p_8^j=p_7^j=p_3^j$ (note that $p_4^0$ has no meaning). An example of where the corollary is used in this way is in \Cref{lemma:T_deriv}.
\end{remark}
\section{Formal differentiability of fixed points}
\label{Appendix:fixed_points}
Like in \Cref{Appendix:composition}, let $(X_s)$ and $(Y_s)$ be decreasing scales of Banach spaces, define $Z_s = X_s \times Y_s$, and fix an index $\mf{s} \in \R$. Suppose that $\mc{X}_{\mf{s}} \subset X_{\mf{s}}$ and $\mc{Y}_{\mf{s}} \subset Y_{\mf{s}}$ are subsets, on which we have a map
\[
    \map{B}{\mc{X}_{\mf{s}} \times \mc{Y}_{\mf{s}}}{X_{\mf{s}}}
\]
with formal derivative
\[
    \map{D^j B}{\mc{X}_{\mf{s}} \times \mc{Y}_{\mf{s}}}{\mc{L}^j\parn[\big]{Z_{\mf{s}-p_4^j}, X_{\mf{s}-p_1^j}}}
\]
for $j=1$. Furthermore, assume that
\begin{align}
    \norm{D^j B[x,y]}_{\mc{L}^j\parn[\big]{Z_{s-p_4^j},X_{s-p_1^j}}}                          & \lesssim \varepsilon,                       &  & s_1^j< s\leq \mf{s},
    \label{eq:app_B_bbd}                                                                                                                                              \\
    \norm{\Delta_h D^j B[x,y]}_{\mc{L}^j\parn[\big]{Z_{s-p_6^j},X_{s-p_2^j}}}                 & \lesssim \varepsilon\norm{h}_{Z_{s-p_5^j}}, &  & s_2^j< s< \mf{s},
    \label{eq:app_B_lip}                                                                                                                                              \\
    \shortintertext{for $j=0,1$, and}
    \norm{\Delta_h D^j B[x,y]-D^{j+1} B[x,y]h}_{\mc{L}^j\parn[\big]{Z_{s-p_8^j},X_{s-p_3^j}}} & = o\parn[\big]{\norm{h}_{Z_{s-p_7^j}}},     &  & s_3^j< s\leq \mf{s},
    \label{eq:app_B_diff}
\end{align}
for $j=0$. Here $p_i^j\geq 0$, and these properties are assumed to hold uniformly for $(x,y),(x,y)+h\in \mc{X}_{\mf{s}}\times \mc{Y}_{\mf{s}}$.

Moreover, assume that $B$ has a fixed point $\map{x}{\mc{Y}_\mf{s}}{\mc{X}_\mf{s}}$; in the sense that
\begin{equation}
    \label{eq:fixed_point}
    B[x[y],y] = x[y]
\end{equation}
for all $y \in \mc{Y}_\mf{s}$. If $\varepsilon$ is sufficiently small, and $p_4^1 = p_1^1$, we can \emph{define} $\map{Dx}{\mc{Y}_\mf{s}}{\mc{L}(Y_{\mf{s}-p_4^1},X_{\mf{s}-p_1^1})}$ through
\begin{equation}
    \label{eq:fixed_point_derivative}
    Dx[y] \ceq (\id-\partial_xB[x[y],y])^{-1} \partial_y B[x[y],y],
\end{equation}
where the inverse can be given by a Neumann series.

\begin{proposition}
    \label{Prop:fixed_point_diff}
    Assume that $p_1^0=0$, $p_4^1=p_1^1$, $p_2^0=p_5^0$, and that $\varepsilon$ is sufficiently small. Then the operators $D^j x$ satisfy:
    \begin{equation}
        \label{eq:app_fix_bdd}
        \norm{D^j x[y]}_{\mc{L}^j\parn[\big]{Y_{s-r_4^j},X_{s-r_1^j}}}\lesssim \varepsilon, \qquad t_1^j< s\leq \mf{s},
    \end{equation}
    for $j=0,1$ if
    \begin{align*}
        r_1^0        & =0,     \\
        t_1^0        & =s_1^0, \\
        \shortintertext{and $j=1$ if}
        r_1^1 =r_4^1 & =p_1^1, \\
        t_1^1        & =s_1^1.
    \end{align*}
    \begin{equation}
        \label{eq:app_fix_lip}
        \norm{\Delta_h D^j x[y]}_{\mc{L}^j\parn[\big]{Y_{s-r_6^j},X_{s-r_2^j}}}\lesssim \varepsilon\norm{h}_{Y_{s-r_5^j}},\qquad t_2^j< s\leq \mf{s},
    \end{equation}
    for $j=0$ if
    \begin{align*}
        r_2^0 = r_5^0 & =p_2^0,                                                                              \\
        t_2^0         & =s_2^0,                                                                              \\
        \shortintertext{and $j=1$ if}
        r_2^1         & =\max\brac{p_1^1,p_2^1,p_2^1-p_5^1+p_2^0,p_2^1-p_6^1+p_1^1},                         \\
        r_5^1         & =r_2^1-p_2^1+p_5^1,                                                                  \\
        r_6^1         & =r_2^1-p_2^1+p_6^1,                                                                  \\
        t_2^1         & =\max\brac{r_2^1+s_1^1-p_1^1,r_2^1+s_2^1-p_2^1,r_5^1+s_2^0-p_2^0,r_6^1+s_1^1-p_1^1}.
    \end{align*}
    \begin{equation}
        \label{eq:app_fix_diff}
        \norm{\Delta_h D^j x[y]-D^{j+1} x[y]h}_{\mc{L}^j\parn[\big]{Y_{s-r_8^j},X_{s-r_3^j}}}= o\parn[\big]{\norm{ h}_{Y_{s-r_7^j}}}, \qquad t_3^j< s\leq \mf{s},
    \end{equation}
    for $j=0$ if
    \begin{align*}
        r_3^0 & =\max\brac{
            p_1^1,
            p_3^0,
            p_3^0-p_7^0+p_2^0
        },                                                                         \\
        r_7^0 & =r_3^0-p_3^0+p_7^0,                                                \\
        t_3^0 & =\max\brac{r_3^0+s_1^1-p_1^1,r_3^0+s_3^0-p_3^0,r_7^0+s_2^0-p_2^0}.
    \end{align*}
\end{proposition}
\begin{proof}
    Since, by definition, $x$ is a fixed point of $B$, we have
    \[
        \norm{x[y] }_{X_s}=\norm{B[x[y],y]}_{X_s} \lesssimop_{\eqref{eq:app_B_bbd}}^{\mathclap{s_1^0 < s \leq \mf{s}}} \varepsilon,
    \]
    proving the first claim for $j=0$. For $j = 1$, we obtain
    \begin{align*}
        \norm{Dx[y]}_{\mc{L}\parn[\big]{Y_{s-p_4^1},X_{s-p_1^1}}} & \leq \parn[\Big]{1-\norm{\partial_xB[x[y],y]}_{\mc{L}\parn[\big]{X_{s-p_1^1}}}}^{-1}\norm{\partial_y B[x[y],y]}_{\mc{L}\parn[\big]{Y_{s-p_4^1},X_{s-p_1^1}}} \\
                                                                  & \lesssimop_{\eqref{eq:app_B_bbd}}^{\mathclap{s_1^1 < s \leq \mf{s}}} (1-\varepsilon)^{-1}\cdot \varepsilon \lesssim \varepsilon
    \end{align*}
    directly from~\eqref{eq:fixed_point_derivative}.

    For the second claim, we note that $\Delta_h x[y] = \Delta_{(\Delta_h x[y],h)}B[x[y],y]$, and therefore
    \[
        \norm{\Delta_h x[y]}_{X_{s-r_2^0}} \qquad \lesssimop_{\eqref{eq:app_B_lip}}^{\mathclap{\substack{r_2^0\geq p_2^0 \\s_2^0+r_2^0-p_2^0 < s \leq \mf{s}}}}\qquad \varepsilon \parn[\big]{\norm{\Delta_h x[y]}_{X_{s-r_2^0}}+\norm{h}_{Y_{s-r_2^0}}},
    \]
    which is sufficient for $j=0$. Similarly, for $j=1$, one may check by direct computation that
    \[
        \Delta_h Dx[y] = \parn*{\id - \partial_x B[x[y],y]}^{-1}\Delta_{(\Delta_hx[y],h)}DB[x[y],y](Dx[y+h],\id),
    \]
    where
    \begin{align*}
        \norm{\parn*{\id - \partial_x B[x[y],y]}^{-1}}_{\mc{L}\parn[\big]{X_{s-r_2^1}}}\qquad                            & \lesssimop_{\mathclap{\eqref{eq:app_B_bbd}}}^{\mathclap{\substack{r_2^1\geq p_1^1               \\s_1^1 + r_2^1 - p_1^1< s \leq \mf{s}}}}\qquad 1,\\
        \norm{\Delta_{(\Delta_hx[y],h)}DB[x[y],y]}_{\mc{L}^j\parn[\big]{Z_{s-r_2^1 + p_2^1 - p_6^1},X_{s-r_2^1}}} \qquad & \lesssimop_{\mathclap{\eqref{eq:app_B_lip}}}^{\mathclap{\substack{r_2^1\geq p_2^1               \\s_2^1 + r_2^1 - p_2^1< s \leq \mf{s}}}}\qquad \varepsilon\parn[\big]{\norm{\Delta_hx[y]}_{X_{s-r_2^1+p_2^1 - p_5^1}} + \norm{h}_{Y_{s-r_2^1+p_2^1 - p_5^1}}},\\
        \qquad                                                                                                           & \lesssimop_{\mathclap{\eqref{eq:app_fix_lip}}}^{\mathclap{\substack{r_2^1\geq p_2^1-p_5^1+p_2^0 \\s_2^0 + r_2^1 - p_2^1 + p_5^1 -p_2^0< s \leq \mf{s}}}}\qquad \varepsilon \norm{h}_{Y_{s-r_2^1+p_2^1 - p_5^1}},\\
        \shortintertext{and}
        \norm{(Dx[y+h]h_1,h_1)}_{Z_{s-r_2^1+p_2^1-p_6^1}} \qquad                                                         & \lesssimop_{\mathclap{\eqref{eq:app_fix_bdd}}}^{\mathclap{\substack{r_2^1\geq p_2^1-p_6^1+p_1^1 \\s_1^1 + r_2^1 - p_2^1+p_6^1 - p_1^1< s \leq \mf{s}}}}\qquad \norm{h_1}_{Y_{s-r_2^1+p_2^1-p_6^1}}.
    \end{align*}

    Finally, for the third claim, we write
    \[
        \Delta_h x[y] - Dx[y]h = \parn[\big]{\id - \partial_x B[x[y],y]}^{-1}\parn[\Big]{\Delta_{(\Delta_h x[y],h)}B[x[y],y] - DB[x[y],y](\Delta_h x[y],h)},
    \]
    whence
    \begin{align*}
        \norm{\Delta_h x[y] - Dx[y]h}_{X_{s-r_3^0}} \qquad & \lesssimop_{\mathclap{\eqref{eq:app_B_bbd}}}^{\mathclap{\substack{r_3^0\geq p_1^1             \\s_1^1 + r_3^0 - p_1^1< s \leq \mf{s}}}}\qquad \norm{\Delta_{(\Delta_h x[y],h)}B[x[y],y] - DB[x[y],y](\Delta_h x[y],h)}_{X_{s-r_3^0}}\\
        \qquad                                             & \equalop_{\mathclap{\eqref{eq:app_B_diff}}}^{\mathclap{\substack{r_3^0\geq p_3^0              \\s_3^0 + r_3^0 -p_3^0< s \leq \mf{s}}}}\qquad o\parn[\big]{\norm{\Delta_h x[y]}_{X_{s-r_3^0 + p_3^0 - p_7^0}} + \norm{h}_{X_{s-r_3^0 + p_3^0 - p_7^0}}}\\
        \qquad                                             & \equalop_{\mathclap{\eqref{eq:app_fix_lip}}}^{\mathclap{\substack{r_3^0\geq p_3^0-p_7^0+p_2^0 \\s_2^0 + r_3^0 - p_3^0+p_7^0 -p_2^0< s \leq \mf{s}}}}\hspace{3.5em} o\parn[\big]{\norm{h}_{X_{s-r_3^0 + p_3^0 - p_7^0}}}. \qedhere
    \end{align*}
\end{proof}
We say that~\eqref{eq:app_B_bbd}--\eqref{eq:app_B_diff} hold for $k$ if they hold for all $j=0,1,\ldots,k-1$, and~\eqref{eq:app_B_bbd} and \eqref{eq:app_B_lip} also hold for $j=k$. Analogous terminology is used for~\eqref{eq:app_fix_bdd}--\eqref{eq:app_fix_diff}.

\begin{corollary}
    \label{Cor:fixedpoint}
    If, in addition to the assumptions of \Cref{Prop:fixed_point_diff}, we have that~\eqref{eq:app_B_bbd}--\eqref{eq:app_B_diff} hold for $k$ with
    \[
        p_i^j =
        \begin{cases}
            i+j-1 & i \leq 3, \\
            1     & i \geq 4
        \end{cases}
        \qquad
        \text{and}
        \qquad
        s_i^j = i + j
    \]
    for all applicable $i$, $j$, then~\eqref{eq:app_fix_bdd}--\eqref{eq:app_fix_diff} hold for $k$ with $r_i^j=p_i^j$ and $t_i^j=s_i^j$.
\end{corollary}
\begin{proof}
    For $k=1$, this follows by direct application of \Cref{Prop:fixed_point_diff}. Assuming the corollary is true for all $k\leq m-1$, where $m \geq 2$, we can define
    \[
        D^m x[y](h_i)_{i \in I_m} = \parn[\big]{\id - \partial_x B[x[y],y]}^{-1} \sum_{\mathclap{\pi \in \mc{P}_m \setminus \brac{I_m}}} D^{\abs{\pi}}B[x[y],y] \parn[\Big]{\parn[\big]{D^{\abs{\beta}}x[y](h_i)_{i \in \beta},h_\beta}}_{\beta \in \pi},
    \]
    where
    \[
        h_\beta \ceq
        \begin{cases}
            h_i & \beta = \brac{i},   \\
            0   & \abs{\beta} \geq 2.
        \end{cases}
    \]
    This is the formal $m$th derivative of $x$, by~\eqref{eq:faa_di_bruno} and \eqref{eq:fixed_point}. Note, crucially, that $\abs{\beta} \leq m-1$ for all $\beta$ appearing on the right-hand side. Repeated use of \Cref{Cor:composition}, and its proof, now shows that the corollary also holds for $k=m$.
\end{proof}

\printbibliography[title=References]
\end{document}